\documentclass[12pt,a4paper]{amsart}
\usepackage{amsmath}
\usepackage{amscd}
\usepackage{amssymb}
\usepackage{drpack}
\usepackage[english]{babel}
\usepackage{verbatim}
\usepackage[shortlabels]{enumitem}
\usepackage{color}
\usepackage{tikz-cd}

\newtheoremstyle{mio}%
{}{} 
{\itshape}{} 
{\bfseries}{.}{ } 
{#1 #2\thmnote{~\mdseries(#3)}} 
\theoremstyle{mio}
\newtheorem{teor}{Theorem}[section]
\newtheorem{cor}[teor]{Corollary}
\newtheorem{prop}[teor]{Proposition}
\newtheorem{lemma}[teor]{Lemma}
\newtheorem{defin}[teor]{Definition}

\newtheoremstyle{definition2}%
{}{} 
{}{} 
{\bfseries}{.}{ } 
{#1 #2\thmnote{\mdseries~ #3}} 
\theoremstyle{definition}
\newtheorem{ex}[teor]{Example}
\newtheorem{oss}[teor]{Remark}

\title{Radical factorization in higher dimension}
\author{Dario Spirito}
\date{\today}
\address{Dipartimento di Scienze Matematiche, Informatiche e Fisiche, Universit\`a degli Studi di Udine, Udine, Italy}
\email{dario.spirito@uniud.it}
\subjclass[2020]{13F05; 13F20; 13C20}
\keywords{Pr\"ufer domains; invertible ideals; radical factorization; SP-domains; free groups; divisorial ideals; integer-valued polynomials; ideal functions.}

\newcommand{\marginparr}[1]{{\color{blue}{{\large$\bullet$}}}\marginpar{\footnotesize{\texttt{#1}}}}

\newcommand{\inscrit}{\mathrm{Crit}}
\newcommand{\Jac}{\mathrm{Jac}}

\newcommand{\V}{\mathcal{V}}
\newcommand{\D}{\mathcal{D}}
\newcommand{\mm}{\mathfrak{m}}

\newcommand{\Int}{\mathrm{Int}}
\newcommand{\Inv}{\mathrm{Inv}}
\newcommand{\Div}{\mathrm{Div}}
\newcommand{\cl}{\mathrm{cl}}
\newcommand{\insfunct}{\mathcal{F}}
\newcommand{\inscont}{\mathcal{C}}
\newcommand{\critx}[1]{\phantom{}^{#1}\inscrit}
\newcommand{\bcrit}{\critx{\omega}}
\newcommand{\inverse}{\mathrm{inv}}

\newcommand{\supp}{\mathrm{supp}}
\newcommand{\Min}{\mathrm{Min}}

\newcommand{\insbound}{\insfunct_b}
\newcommand{\funcontcomp}{\mathcal{C}_c}

\begin{document}

\begin{abstract}
We generalize the theory of radical factorization from almost Dedekind domain to strongly discrete Pr\"ufer domains; we show that, for a fixed subset $X$ of maximal ideals, the finitely generated ideals with $\V(I)\subseteq X$ have radical factorization if and only if $X$ contains no critical maximal ideals with respect to $X$. We use these notions to prove that in the group $\Inv(D)$ of the invertible ideals of a strongly discrete Pr\"ufer domains is often free: in particular, we show it when the spectrum of $D$ is Noetherian or when $D$ is a ring of integer-valued polynomials on a subset over a Dedekind domain.
\end{abstract}

\maketitle

\section{Introduction}
Let $D$ be an integral domain. An ideal $I$ of $D$ is said to have \emph{radical factorization} if $I$ can be written as a finite product of radical ideals, and $D$ is said to be an \emph{SP-domain} if every ideal has radical factorization. Radical factorization is a generalization of prime factorization and, consequently, SP-domains are a generalization of Dedekind domain; nevertheless, SP-domains are rather close to Dedekind domains, since every SP-domain is an \emph{almost Dedekind domain}, i.e., it is locally a Dedekind domain or, equivalently, it is locally a discrete valuation ring (DVR).

An interesting feature of SP-domains is that their invertible ideals can be represented by continuous functions from the maximal space (endowed with the inverse topology, see below) to $\insZ$ (cfr. \cite[Theorem 5.1]{HK-Olb-Re} and \cite[Corollary 4.2]{SP-scattered}); more generally, using SP-domains as a starting point, is it possible to prove that, if $D$ is an almost Dedekind domain, the group $\Inv(D)$ of the invertible ideals of $D$ is isomorphic to a direct sum of group of continuous functions, and in particular it is free (\cite[Theorem 5.9]{SP-scattered} and \cite[Proposition 5.3]{bounded-almded}). This result suggested the conjecture, advanced in \cite{inv-free}, that $\Inv(D)$ is free for all strongly discrete Pr\"ufer domains $D$; the aforementioned result about almost Dedekind domains represents the one-dimensional case of the conjecture. The aim of this paper is to extend the theory of radical factorization and the methods used in the almost Dedekind domain case to domains of higher dimensions.

Since an SP-domain is an almost Dedekind domain, it has dimension $1$. However, if $V$ is a valuation domain with principal maximal ideal $\mm$, then every ideal $I$ such that $\dim(V/I)=0$ (i.e., every ideal having $\mm$ as minimal prime) is in the form $\mm^n$, and thus in particular has radical factorization; that is, $V$ has the SP-property if we restrict to ideals of coheight $0$. The main idea of this paper is to generalize this idea by focusing on finitely generated ideals whose associated closed set $\V(I)$ is contained in a subset $X$ of the maximal space $\Max(D)$. We show that, in this setting, it is possible to generalize the notions used in the theory of almost Dedekind domains to study radical factorization: in particular, it makes sense to define the notion of critical ideals (relative to $X$), and the domain $D$ has the SP-property relative to $X$ (i.e., every ideal $I$ with $\V(I)\subseteq X$ has radical factorization) if and only if $X$ contains no critical ideals.

More generally, we focus on the group $\Inv_X(D)$ of the fractional ideals whose support is contained in $X$ (the support is a generalization of the closed set $\V(I)$ associated to an ideal; see Definition \ref{def:support}): these groups can be seen as a far-reaching generalization of the group of invertible unitary ideals used in the study of the ring of integer-valued polynomials and its generalizations (see \cite[Chapter VIII]{intD} and \cite{locpic}). To study $\Inv_X(D)$ and to obtain from groups of this kind information about the full group $\Inv(D)$ of the invertible ideals of $D$, we use two transfinite constructions.

To study $\Inv_X(D)$, we use the same method used in \cite{SP-scattered} and \cite{bounded-almded} in the context of almost Dedekind domains to set up a chain of subsets of $X$, associated to a chain of overrings of $D$, such that at every step of the chain we isolate the ideals with radical factorization. In particular, we obtain that, if $X\subseteq\Max(D)$ is open with respect to the inverse topology, $\Inv_X(D)$ can be written as a direct sum of a family of groups of continuous functions from subsets of $X$ to $\insZ$, and thus it is free (Theorem \ref{teor:Invd}); if $X$ has no critical ideals, $\Inv_X(D)$ is isomorphic to the group of continuous functions with compact support from $X$ to $\insZ$ (Corollary \ref{cor:invX-sp}).

To pass from $\Inv_X(D)$ to $\Inv(D)$, we show that their quotient is isomorphic to the group $\Inv(\Omega(X))$ of invertible ideals of the Nagata transform $\Omega(X)$ of $X$, which is an overring of $D$; we then introduce the concept of a ring whose spectrum is \emph{compactly layered} in order to apply the same reasoning to $\Omega(X)$ and to a family of overrings obtained inductively. This class is quite general and includes, beyond one-dimensional domains, all Pr\"ufer domains with Noetherian spectrum as well as the ring $\Int(D)$ of integer-valued polynomials when $D$ is a Dedekind domain with finite residue fields. We show that, under this hypothesis, $\Inv(D)$ can be expressed as a direct sum of groups $\Inv_{X_\alpha}(\Omega(Y_\alpha))$, where the $X_\alpha$s and the $Y_\alpha$s are subsets of $\Spec(D)$: it follows that, for strongly discrete Pr\"ufer domains with compactly layered spectrum, the group $\Inv(D)$ is free.

We also study divisorial ideals: in particular, we show that the completion of $\Inv_X(D)$ as an $\ell$-group is the group $\Div_X(D)$ of the $v$-invertible divisorial ideals $I$ such that the support of $I$ and $(D:I)$ is contained in $X$, extending the same result for $\Inv(D)$ and $\Div(D)$ in the case of one-dimensional (or, more generally, completely integrally closed) Pr\"ufer domains \cite[Proposition 3.1]{HK-Olb-Re}. In particular, when $X$ has no critical ideals, we show that $\Div_X(D)$ and $\Div_X(D)/\Inv_X(D)$ are free, as it happens for SP-domains (Proposition \ref{prop:Div-SP}).

\section{Preliminaries}
Throughout the paper, all rings are commutative and unitary, and they will almost always be integral domains. If $D$ is an integral domain with quotient field $K$, an \emph{overring} of $D$ is a ring $T$ such that $D\subseteq T\subseteq K$. We denote by $\Over(D)$ the set of overrings of $D$.

\subsection{Fractional and invertible ideals}
Let $D$ be an integral domain with quotient field $K$. A \emph{fractional ideal} of $D$ is a $D$-submodule $I$ of $K$ such that $dI\subseteq D$ for some $d\in D$, $d\neq 0$; a fractional ideal contained in $D$ is just an ideal of $D$ (and we sometimes call them \emph{integral} ideals for emphasis).

A fractional ideal $I$ is an \emph{invertible ideal} if there is a fractional ideal $J$ such that $IJ=D$; in this case, $J$ is equal to the conductor
\begin{equation*}
I^{-1}:=(D:I):=\{x\in K\mid xI\subseteq D\}.
\end{equation*}
Every invertible ideal is finitely generated. The set of invertible ideals is a group under product, denoted by $\Inv(D)$.

If $I$ is a fractional ideal, the $v$-closure of $I$ is $I^v:=(D:(D:I))$; equivalently, $I^v$ is equal to the intersection of all principal fractional ideals containing $I$. An ideal $I$ is \emph{divisorial} if $I=I^v$. An ideal $I$ is \emph{$v$-invertible} if $(IJ)^v=D$ for some fractional ideal $J$, or equivalently if $(I:I)=D$ \cite[Proposition 34.2]{gilmer}. The set $\Div(D)$ of $v$-invertible divisorial ideals is a group under the ``$v$-product'' $I\times_vJ:=(IJ)^v$ (this follows, for example, from the general case of star operations: see \cite{jaffard_systeme,griffin_vmultiplication_1967,zafrullah_tinvt,halterkoch_libro}).

\subsection{Valuation and Pr\"ufer domains}
Let $V$ be an integral domain with quotient field $K$. Then, $V$ is said to be a \emph{valuation domain} if, for every $x\in K$, at least one of $x$ and $x^{-1}$ is in $V$; to each valuation domain can be associated a surjective map $v:K\setminus\{0\}\longrightarrow\Gamma_v$, where $\Gamma_v$ is a totally ordered group, such that $v(x+y)\geq\min\{v(x),v(y)\}$ and $v(xy)=v(x)+v(y)$. If $\Gamma_v\simeq\insZ$, then $V$ is said to be a \emph{discrete valuation domain} or a \emph{discrete valuation ring} (DVR).

An integral domain $D$ is a \emph{Pr\"ufer domain} if it is locally a valuation domain, i.e., if $D_P$ is a valuation domain for every $P\in\Spec(D)$. Among the many characterizations of Pr\"ufer domains (see e.g. \cite[Chapter IV]{gilmer} or \cite[Theorem 1.1.1]{fontana_libro}), we have that $D$ is a Pr\"ufer domain if and only if every finitely generated fractional ideal is invertible.

A Pr\"ufer domain $D$ is \emph{strongly discrete} if no nonzero prime ideal of $D$ is idempotent, i.e., if $P\neq P^2$ for all $P\in\Spec(D)$, $P\neq(0)$. Equivalently, $D$ is strongly discrete if and only if $PD_P$ is principal (as an ideal of $D_P$) for all $P\in\Spec(D)$ \cite[Section 5.3]{fontana_libro}.

\subsection{$\ell$-groups}
An $\ell$-group is an ordered group whose order structure is a lattice. We refer to \cite{darnel-lgroups} for results about $\ell$-groups.

A subgroup $H$ of an $\ell$-group $G$ is \emph{convex} if when $a<b<c$ and $a,c\in H$, then also $b\in H$. An $\ell$-group $G$ is \emph{archimedean} if, whenever $a,b\in G$ satisfy $na\leq b$ for all $n\inN$, then $a\leq 0$.

An $\ell$-group is \emph{complete} if any subset that is bounded above has a supremum (equivalently, if any subset that is bounded below has an infimum). Every $\ell$-group $G$ has a completion, i.e., there is a complete $\ell$-group $\widehat{G}$ such that $G$ is a dense subgroup of $\widehat{G}$; the completion $\widehat{G}$ is the minimal complete $\ell$-group containing $G$, in the sense that every map from $G$ to a complete $\ell$-group $H$ factors through $\widehat{G}$.

If $D$ is a Pr\"ufer domain, the sum and the intersection of two invertible ideals is invertible, and thus the group $\Inv(D)$ is an $\ell$-group when the order is the reverse containment. If $D$ is one-dimensional, or more generally if $D$ is completely integrally closed, $\Inv(D)$ is archimedean, and $\Div(D)$ is the completion of $\Inv(D)$ \cite[Proposition 3.1]{HK-Olb-Re}.

\subsection{Groups of function}
An abelian group $G$ is \emph{free} if it has a basis, i.e., if there is an $E\subset G$ such that every element of $G$ can be written uniquely as a sum of elements of $E$; equivalently, $G$ is free if it is isomorphic to a direct sum of cyclic infinite groups. Every subgroup of a free group is free.

Let $X$ be a set. We denote by $\insfunct(X,\insZ)$ the set of functions from $X$ to $\insZ$ and by $\insfunct_b(X,\insZ)$ the set of bounded functions from $X$ to $\insZ$. Both these sets are groups under the componentwise sum of functions, and $\insfunct_b(X,\insZ)$ is free.

A subgroup $G\subseteq\insfunct_b(X,\insZ)$ is a \emph{Specker group} if, for every $f\in G$, the characteristic function $\chi_{f^{-1}(n)}$ belongs to $G$ for every $n\in\insZ$ \cite[\textsection 1]{nobeling}. If $G$ is a Specker group, then $G$ is a direct summand of $\insfunct_b(X,\insZ)$, i.e., there is a subgroup $H$ such that $\insfunct_b(X,\insZ)=G\oplus H$; in particular, the quotient $\insfunct_b(X,\insZ)/G$ is free \cite[Satz 2]{nobeling}.

If $X$ is a topological space, we denote by $\inscont(X,\insZ)$ the group of continuous functions $X\longrightarrow\insZ$, where $\insZ$ is endowed with the discrete topology; we denote by $\inscont_c(X,\insZ)$ the subgroup of continuous function with compact support, where the \emph{support} $\supp(f)$ of $f$ is the closure of its cozero set $f^{-1}(\insZ\setminus\{0\})$. Every continuous function of compact support is bounded, and thus $\inscont_c(X,\insZ)$ is free, being a subgroup of $\insfunct_b(X,\insZ)$.

The groups $\inscont(X,\insZ)$ and $\inscont_c(X,\insZ)$ are $\ell$-groups, where $f\leq g$ if $f(x)\leq g(x)$ for every $x\in X$. When $X$ is an Hausdorff regular space, their completions are, respectively, the groups $\inscont(E_X,\insZ)$ and $\inscont_c(E_X,\insZ)$, where $E_X$ is the \emph{Gleason cover} of $X$ (see \cite{strauss-extremallydisconnected} for the definition and construction of the Gleason cover and \cite[Lemma 5.2 and Theorem 5.3]{HK-Olb-Re} and \cite[Lemma 6.4]{SP-scattered} for the group-theoretic completion).

\subsection{Topologies on the spectrum}
Let $R$ be a commutative ring. The Zariski topology on the spectrum $\Spec(R)$ of $R$ is the topology whose closed sets are the sets $\V(I):=\{P\in\Spec(R)\mid I\subseteq P\}$, as $I$ ranges among the ideals of $R$.

Given $X\subseteq\Spec(R)$, we denote by $X^\uparrow$ the closure under specialization of $X$, i.e., $X^\uparrow:=\{P\in\Spec(R)\mid Q\subseteq P$ for some $Q\in X\}$; likewise, the closure under generization of $X$ is $X^\downarrow:=\{P\in\Spec(R)\mid Q\supseteq P$ for some $Q\in X\}$.

The \emph{inverse topology} on $\Spec(R)$ is the topology for which a subbasis is the family of the sets $\V(I)$, as $I$ ranges among the finitely generated ideals of $R$. The closed sets of $\Spec(R)$ with respect to the inverse topology are exactly the subsets $X$ that are closed by generization (i.e., $X=X^\downarrow$) and compact with respect to the Zariski topology. We use $\cl_\inverse(X)$ to denote the closure of $X\subseteq\Spec(R)$ with respect to the inverse topology.

The \emph{constructible topology} on $\Spec(R)$ is the topology generated by the Zariski and the inverse topology (i.e., the coarsest topology that is finer both of the Zariski and the inverse topology). 

The spectrum of $R$ is compact with respect to all three topologies; under the constructible topology, it is also Hausdorff.

The inverse and the constructible topology agree on the maximal space $\Max(R)$ \cite[Corollary 4.4.9(i)]{spectralspaces-libro}; moreover, if $X=\V(I)\subseteq\Max(R)$ for some ideal $I$, then these two topology agree on $X$ also with the Zariski topology, since the three topologies agree for spectra of dimension $0$ (see \cite[Chapter 1]{spectralspaces-libro}).

\section{Invertible ideals with given support}
In this section, we introduce and study the basic properties of the group $\Inv_X(D)$ for some $X\subseteq\Spec(D)$, which will be the main tool of our paper. To do so, we need to associate to each subset of the spectrum an overring.

\begin{defin}
Let $D$ be a domain and let $X\subseteq\Spec(D)$. The \emph{Nagata transform} of $D$ with respect to $X$ is
\begin{equation*}
\Omega(X):=\bigcap\{D_P\mid P\in\Spec(D)\setminus X^\uparrow\}.
\end{equation*}
\end{defin}

\begin{oss}
The terminology ``Nagata transform'' is chosen to generalize the notion of Nagata transform $\Omega(I)$ of an ideal $I$. Indeed, if $X=\V(I)$ for some ideal $I$, then $\Omega(X)=\Omega(I)$ \cite[Theorem 3.2.2]{fontana_libro}.
\end{oss}

\begin{prop}\label{prop:splitting}
Let $D$ be a Pr\"ufer domain and let $X=X^\uparrow\subseteq\Spec(D)$. Then, the following are equivalent:
\begin{enumerate}[(i)]
\item\label{prop:splitting:compact} $\Spec(D)\setminus X$ is compact;
\item\label{prop:splitting:inverse} $X$ is open, with respect to the inverse topology;
\item\label{prop:splitting:explode} $P\Omega(X)=\Omega(X)$ for every $P\in X$.
\end{enumerate}
\end{prop}
\begin{proof}
Let $Y:=\Spec(D)\setminus X$.

\ref{prop:splitting:compact} $\iff$ \ref{prop:splitting:inverse} Since $X=X^\uparrow$, we have $Y=Y^\downarrow$ and thus $Y$ is compact if and only if it is closed, with respect to the inverse topology. The equivalence follows. 

\ref{prop:splitting:explode} $\iff$ \ref{prop:splitting:inverse} The prime ideals of $D$ that survive in $\Omega(X)$ are exactly the ones in the closure of $Y$, with respect to the inverse topology (use \cite[Corollary 4.10]{fifolo_transactions}, using the fact that a Pr\"ufer domain is vacant \cite[Remark 2.2]{vacantdomains} and the Zariski space of a Pr\"ufer domain consists of the localizations of $D$). Hence, $X$ is open (i.e., $Y$ is closed) if and only if every $P\in X$ extends to the whole $\Omega(X)$.
\end{proof}

\begin{defin}
Let $D$ be a Pr\"ufer domain. We say that a subset $X\subseteq\Spec(D)$ is a \emph{splitting set} if $X=X^\uparrow$ and if it satisfies one (hence all) conditions of Proposition \ref{prop:splitting}.

We also say that a splitting set is \emph{discrete} if no $P\in X$ is idempotent, i.e., if $P\neq P^2$ for every $P\in X$.
\end{defin}

\begin{defin}\label{def:support}
Let $D$ be an integral domain, and let $I$ be a fractional ideal of $D$. The \emph{support} of $I$ is
\begin{equation*}
\supp(I):=\{P\in\Spec(D)\mid ID_P\neq D_P\}.
\end{equation*}
\end{defin}

Note that, for a proper ideal $I$ of $D$, the support of $I$ is just the closed set $\V(I)$ associated to $I$.

\begin{lemma}\label{lemma:supp-prodotto}
Let $D$ be a Pr\"ufer domain and let $I,J$ be fractional ideals of $D$. Then, $\supp(IJ)\subseteq\supp(I)\cup\supp(J)$.
\end{lemma}
\begin{proof}
If $P\in\supp(IJ)$, then $D_P\neq(IJ)D_P=(ID_P)(JD_P)$. Thus, at least one of $ID_P$ and $JD_P$ is different from $D_P$, i.e., $P\in\supp(I)\cup\supp(J)$.
\end{proof}

\begin{lemma}\label{lemma:JL-1}
Let $D$ be a Pr\"ufer domain, and let $I$ be a finitely generated fractional ideal of $D$. Then, there are ideals $J,L\subseteq D$ such that $\supp(J),\supp(L)\subseteq\supp(I)$ and $I=JL^{-1}$.
\end{lemma}
\begin{proof}
Since $D$ is a Pr\"ufer domain, we have $I=ID=(I\cap D)(I+D)$ \cite[Theorem 25.2(c)]{gilmer}. We have 
\begin{equation*}
\begin{aligned}
\supp(I\cap D)&=\{P\in\Spec(D)\mid (I\cap D)D_P\neq D_P\}=\\
&=\{P\in\Spec(D)\mid ID_P\cap D_P\neq D_P\}=\\
&=\{P\in\Spec(D)\mid ID_P\subsetneq D_P\}\subseteq\supp(I)
\end{aligned}
\end{equation*}
since $D_P$ is a valuation ring. Likewise, $\supp(I+D)=\{P\in\Spec(D)\mid ID_P\supsetneq D_P\}\subseteq\supp(I)$. Hence, both $I\cap D$ and $I+D$ have support inside $\supp(I)$; moreover, $D\subseteq(I+D)$, and thus $L:=(I+D)^{-1}\subseteq D$. Therefore, setting $J:=I\cap D$ we have $I=JL^{-1}$.
\end{proof}

\begin{defin}
Let $D$ be a Pr\"ufer domain and let $X$ be a splitting set of $D$. Then, the set of \emph{$X$-invertible ideals} of $D$ is
\begin{equation*}
\Inv_X(D):=\{I\in\Inv(D)\mid \supp(I)\subseteq X\}.
\end{equation*}
\end{defin}

\begin{prop}\label{prop:kerext}
Let $D$ be a Pr\"ufer domain and let $X$ be a splitting set of $D$. Then, $\Inv_X(D)$ is a convex subgroup of $\Inv(D)$ and
\begin{equation*}
\frac{\Inv(D)}{\Inv_X(D)}\simeq\Inv(\Omega(X)).
\end{equation*}
If also $X\subseteq\Max(D)$, then $\Inv_X(D)$ is archimedean.
\end{prop}
\begin{proof}
If $I,J\in\Inv_X(D)$, then $\supp(IJ)\subseteq\supp(I)\cup\supp(J)\subseteq X$ by Lemma \ref{lemma:supp-prodotto}, and thus $IJ\in\Inv_X(D)$. Moreover, if $I\in\Inv_X(D)$ then $\supp(I^{-1})=\supp(I)\subseteq X$, and thus $\Inv_X(D)$ is stable by taking inverse. Hence $\Inv_X(D)$ is a subgroup of $\Inv(D)$.

Suppose $I\subseteq J\subseteq L$ and $I,L\in\Inv_X(D)$. If $P\in\supp(J)$, then $JD_P\neq D_P$ and thus either $JD_P\subsetneq D_P$ or $D_P\subsetneq JD_P$: in the former case $P\in\supp(I)\subseteq X$, in the latter $P\in\supp(L)\subseteq X$. Hence $\supp(J)\subseteq X$ and thus $J\in\Inv_X(D)$, i.e., $\Inv_X(D)$ is convex.

Consider now the extension map $\lambda:\Inv(D)\longrightarrow\Inv(\Omega(X))$; since $D$ is Pr\"ufer, $\lambda$ is surjective. We claim that $\ker\lambda=\Inv_X(D)$. Indeed, we have that $I\Omega(X)=\Omega(X)$ if and only if $P\Omega(X)=\Omega(X)$ for every ideal $P\in\supp(I)$; since $X$ is a splitting set, by Proposition \ref{prop:splitting} this is equivalent to $\supp(I)\subseteq X$. Hence $\ker\lambda=\Inv_X(D)$.

Suppose now that $X\subseteq\Max(D)$: we show that $\Inv_X(D)$ is archimedean. Suppose that $J\subseteq I^n$ for some $I,J\in\Inv_X(D)$. Then, $J\cap D\subseteq I^n\cap D=(I\cap D)^n$ (where the last equality comes from \cite[Theorem 2.2]{gilmer-grams}); if $I\cap D\neq D$, then $I\cap D\subseteq P$ for some prime ideal $P$, and thus $J\cap D\subseteq\bigcap_{n\geq 1}P^n=Q$ is a prime ideal properly contained in $P$ \cite[Theorem 23.3]{gilmer}. Hence $Q\in\supp(J)$; however, $Q$ is not maximal, against the fact that $\supp(J)\subseteq X\subseteq\Max(D)$. Thus $I\cap D=D$, i.e., $D\subseteq I$. Hence $\Inv_X(D)$ is archimedean.

\end{proof}

\section{Divisorial ideals}
In this section, we construct an analogue for divisorial ideals of the group $\Inv_X(D)$. In particular, we show that if $X\subseteq\Max(D)$ the relationship between $\Inv_X(D)$ and $\Div_X(D)$ closely mirrors the relationship that $\Inv(D)$ and $\Div(D)$ have when $D$ is one-dimensional.

\begin{defin}
Let $D$ be a Pr\"ufer domain and let $X$ be a splitting set. We define $\Div_X(D)$ be to the set of all $I\in\Div(D)$ such that $\supp(I),\supp(D:I)\subseteq X$.
\end{defin}

\begin{oss}
The fact that $\supp(I)\subseteq X$ does not imply that $\supp(D:I)\subseteq X$. For example, if $D$ is a valuation domain of dimension $2$ and $\Spec(D)=\{(0)\subset Q\subset P\}$, then $I=D_Q$ is a fractional ideal such that $\supp(D_Q)=\{P\}$, but $(D:D_Q)=Q$ and $\supp(Q)=\{Q,P\}$. However, $D_Q$ is not $v$-invertible (since $(QD_Q)^v=QD_Q$), i.e., $D_Q\notin\Div(D)$; thus it is not clear if the assumption $\supp(D:I)\subseteq X$ can be omitted from the previous definition.
\end{oss}

The following lemma generalizes the fact that, in a one-dimensional Pr\"ufer domain, every divisorial ideal is $v$-invertible (cfr. \cite[Theorem 34.3]{gilmer}).

\begin{lemma}\label{lemma:cic}
Let $D$ be a Pr\"ufer domain. If $I$ is a proper ideal of $D$ such that $\V(I)\subseteq\Max(D)$, then $I$ is $v$-invertible.
\end{lemma}
\begin{proof}
Let $M$ be a maximal ideal of $D$. If $M\notin\V(I)$, then $(ID_M:ID_M)=(D_M:D_M)=D_M$; if $M\in\V(I)$, then $ID_M$ is an ideal of $D_M$ whose minimal prime is the maximal ideal $MD_M$, and thus $(ID_M:ID_M)=D_M$. Therefore,
\begin{equation*}
\begin{aligned}
D\subseteq(I:I)= & \bigcap_{M\in\Max(D)}(I:I)D_M\subseteq\\
\subseteq & \bigcap_{M\in\Max(D)}(ID_M:ID_M)=\bigcap_{M\in\Max(D)}D_M=D,
\end{aligned}
\end{equation*}
and so $I$ is $v$-invertible.
\end{proof}

\begin{lemma}\label{lemma:densita-basso}
Let $D$ be a Pr\"ufer domain and let $X\subseteq\Max(D)$ be a splitting set. If $I$ is a proper ideal of $D$ such that $\V(I)\subseteq X$, then there is a $J\in\Inv_X(D)$ such that $J\subseteq I$.
\end{lemma}
\begin{proof}
Let $Q\in\Spec(D)\setminus X$. Then, $ID_Q=D_Q$, and thus $J_QD_Q=D_Q$ for some finitely generated ideal $J_Q\subseteq I$; therefore, $J_Q\nsubseteq Q$, i.e., $Q\in\D(J_Q)$. It follows that $\{\D(J_Q)\mid Q\in\Spec(D)\setminus X\}$ is an open cover of $\Spec(D)\setminus X$, which is compact in the Zariski topology by Proposition \ref{prop:splitting}. Therefore, there is a finite subcover $\D(J_1),\ldots,\D(J_n)$; in particular, $J:=J_1+\cdots+J_n$ is a finitely generated ideal such that $JD_Q=D_Q$ for every $Q\notin X$. Hence, $\supp(J)\subseteq X$ and $J\in\Inv_X(D)$. The claim is proved.
\end{proof}

The following theorem partially generalizes \cite[Proposition 3.1]{HK-Olb-Re}.
\begin{teor}\label{teor:DivX}
Let $D$ be a Pr\"ufer domain and let $X\subseteq\Max(D)$ be a splitting set. Then, $\Div_X(D)$ is a convex subgroup of $\Div(D)$; moreover, $\Div_X(D)$ is a complete $\ell$-group and is the completion of $\Inv_X(D)$.
\end{teor}
\begin{proof}
If $J\in\Div_X(D)$ then its inverse $(D:J)$ is in $\Div_X(D)$ too, as $(D:(D:J))=J$. Let $I,J\in\Div_X(D)$, and let $P\in\Spec(D)\setminus X$: then, $IJD_P=D_P$ and thus $D_P\subseteq(IJ)^vD_P$. Likewise, we also have $D_P\subseteq((D:I)(D:J))^vD_P$; since $((D:I)(D:J))^v$ is the inverse of $(IJ)^v$ in $\Div(D)$, it follows that both containments are equalities. Thus $\supp((IJ)^v),\supp(D:(IJ)^v)\subseteq X$ and $\Div_X(D)$ is closed by the $v$-product. Hence $\Div_X(D)$ is a subgroup of $\Div(D)$.

If $I\subseteq J\subseteq D$ and $I\in\Div_X(D)$, then $\supp(J)\subseteq\supp(I)\subseteq X$; likewise, $D\subseteq(D:J)\subseteq(D:I)$ and thus $\supp(D:J)\subseteq\supp(D:I)\subseteq X$. Hence $J\in\Div_X(D)$ and $\Div_X(D)$ is convex.

Let $\{J_\alpha\}_{\alpha\in A}$ be a subset of $\Div_X(D)$ that is bounded below, i.e., such that $I\subseteq J:=\bigcap_\alpha J_\alpha$ for some $I\in\Div_X(D)$. We can suppose without loss of generality that $J_\alpha\subseteq D$, and thus both $I$ and $J$ are in $D$. By construction, $J$ is divisorial and $\supp(J)\subseteq\supp(I)\subseteq X$; since $X\subseteq\Max(D)$, by Lemma \ref{lemma:cic} $J$ is also $v$-invertible, i.e., $J\in\Div(D)$. Since $\Div_X(D)$ is convex we have $J\in\Div_X(D)$, and thus $\Div_X(D)$ is complete.

Finally, we show that $\Inv_X(D)$ is dense in $\Div_X(D)$: given a proper ideal $I\in\Div_X(D)$, we need $J,L\in\Inv_X(D)$ such that $J\subseteq I\subseteq L\subsetneq D$. Since $I$ is divisorial and properly contained in $D$, there is an $x\in K$ such that $I\subseteq D\cap xD\subsetneq D$; thus we can take $L=D\cap xD$, and $L\in\Inv_X(D)$ since $\supp(L)\subseteq\supp(J)\subseteq X$. To find $I$ it is enough to apply Lemma \ref{lemma:densita-basso}.
\end{proof}

\section{Critical ideals}\label{sect:critical}
In this and the following section we show how the results about invertible ideals in an almost Dedekind domains can be generalized to the set $\Inv_X(D)$ when $X\subseteq\Max(D)$ is a splitting set. To do so, we need to construct ideal functions from $X$ to $\insZ$.

Let $P$ be a nonidempotent prime ideal of the Pr\"ufer domain $D$. Then, the set $Q:=\bigcap_{n\geq 1}P^n$ is a prime ideal strictly contained in $P$, and is actually the largest prime ideal strictly contained in $P$. The quotient $D_P/QD_P$ is a one-dimensional valuation domain with nonidempotent maximal ideal, and thus it is a discrete valuation ring; hence, it carries a valuation $v'_P$ with values in $\insZ$. If $P$ is minimal over an ideal $I$, we have $QD_P\subsetneq ID_P$, and thus the number
\begin{equation*}
w_P(I):=v'_P(ID_P/Q_DP)
\end{equation*}
is well-defined.

More generally, the previous reasoning holds whenever $I$ is a fractional ideal such that $QD_P\subsetneq ID_P\subsetneq D_Q$; this observation prompts the following definition.
\begin{defin}
Let $D$ be a Pr\"ufer domain and let $X\subseteq\Max(D)$. A fractional ideal $I$ is \emph{evaluable on $X$} if, for every $P\in X$, we have $QD_P\subsetneq ID_P\subsetneq D_Q$, where $Q=\bigcap_{n\geq 1}P^n$.
\end{defin}
The previous reasoning can thus be rephrased as saying that, if $I$ is evaluable on $X$, the \emph{ideal function}
\begin{equation*}
\begin{aligned}
\nu_I\colon X & \longrightarrow \insZ,\\
P & \longmapsto w_P(I)
\end{aligned}
\end{equation*}
is well-defined; in particular, $\nu_I$ will exists for all proper ideals $I$ with $\V(I)\subseteq X$.

\begin{lemma}\label{lemma:prod-evaluable}
Let $X\subseteq\Max(D)$. If $I,J$ are fractional ideals that are evaluable on $X$, then so is $IJ$, and $\nu_{IJ}=\nu_I+\nu_J$.
\end{lemma}
\begin{proof}
Let $P\in X$ and $Q:=\bigcap_{n\geq 1}P^n$. Since $QD_P\subsetneq ID_P,JD_P$ are $QD_P$ is prime, $IJD_P$ is not contained in $QD_P$ and thus $QD_P\subsetneq IJD_P$ (as $D_P$ is a valuation domain). Moreover, $IJD_P\subsetneq D_Q$ since $D_Q$ is a ring and $ID_P,JD_P\subsetneq D_Q$. Since $P$ is arbitrary, $IJ$ is evaluable on $X$. The last claim is obvious.
\end{proof}

\begin{lemma}\label{lemma:nuI-InvX}
Let $X\subseteq\Max(D)$ be a discrete splitting set. Every $I\in\Div_X(D)$ is evaluable on $X$.
\end{lemma}
\begin{proof}
Let $P\in X$ and $Q:=\bigcap_{n\geq 1}P^n$. 

Suppose first that $I\in\Inv_X(D)$. We need to show that $QD_Q\subsetneq ID_P\subsetneq D_Q$. If $P\notin\supp(I)$, then $ID_P=D_P$ and we are done. If $P\in\supp(I)$, then $Q\notin\supp(I)$ since $\supp(I)\subseteq\Max(D)$, and thus $ID_Q=D_Q$; hence, $QD_Q\subsetneq ID_P\subseteq D_Q$. Since $I$ is finitely generated, $ID_P$ is principal over $D_P$, and thus it can't be equal to $D_Q$; the claim follows.

Suppose that $I\in\Div_X(D)$. Then, $I=(I\cap D)(I+D)$ \cite[Theorem 25.2(d)]{gilmer} and thus
\begin{equation*}
I=I^v=((I\cap D)(I+D))^v=((I\cap D)^v(I+D)^v)^v=(I\cap D)\times_v(I+D)^v.
\end{equation*}
Since $(I+D)^v$ is the inverse in $\Div_X(D)$ of a proper ideal of $D$, it is enough to prove that $J$ and $(D:J)$ are evaluable on $X$ for every proper ideal $J\in\Div_X(D)$. By Lemma \ref{lemma:densita-basso}, there is an $L\in\Inv_X(D)$ such that $L\subseteq J$. Hence, for every $P\in X$ we have $QD_Q\subsetneq LD_P\subseteq JD_P\subseteq D_P\subsetneq D_Q$; likewise, $D_P\subseteq (D:J)D_P\subseteq(D:L)D_P\subsetneq D_Q$. Hence both $J$ and $(D:J)$ are evaluable on $X$, and the claim is proved.
\end{proof}

Lemma \ref{lemma:nuI-InvX} implies that there is a map
\begin{equation*}
\begin{aligned}
\Psi\colon\Div_X(D) &\longrightarrow \insfunct(X,\insZ),\\
I & \longmapsto \nu_I.
\end{aligned}
\end{equation*}
We want to understand the properties of this map, and in particular its image; we focus in this section on $\Inv_X(D)$, while leaving some consequences about $\Div_X(D)$ to Section \ref{sect:DivX-2}.

\begin{lemma}\label{lemma:prop-psi}
Preserve the notation above. Then:
\begin{enumerate}[(a)]
\item\label{lemma:prop-psi:inj} $\Psi$ is injective;
\item\label{lemma:prop-psi:Inv} the restriction of $\Psi$ to $\Inv_X(D)$ is a group homomorphism.
\end{enumerate}
\end{lemma}
\begin{proof}
\ref{lemma:prop-psi:inj} Suppose $\nu_I=\nu_J$. If $I\neq J$, there is a maximal ideal such that $ID_M\neq JD_M$; in particular, $M$ must belong to at least one between $\supp(I)$ and $\supp(J)$, and thus to $X$. But then, $\nu_I(M)=\nu_J(M)$, which forces $ID_M=JD_M$, a contradiction.

\ref{lemma:prop-psi:Inv} follows from Lemma \ref{lemma:nuI-InvX}.
\end{proof}

\begin{oss}\label{oss:Psi}
When considered as a map from $\Div_X(D)$ to $\insfunct(X,\insZ)$, $\Psi$ is \emph{not} a group homomorphism. Indeed, if $I,J\in\Div_X(D)$ then $\Psi(I)\Psi(J)=\nu_I+\nu_J=\nu_{IJ}$, but $\Psi(I\times_vJ)=\nu_{I\times_vJ}=\nu_{(IJ)^v}$, and if $IJ\neq(IJ)^v$ then the two ideal functions will be distinct.
\end{oss}

\begin{prop}\label{prop:semicont}
Let $D$ be a Pr\"ufer domain and let $X\subseteq\Max(D)$ be a discrete splitting set of $D$. For every $I\in\Inv_X(D)$, the ideal function $\nu_I$ is lower semicontinuous, with respect to the inverse topology.
\end{prop}
\begin{proof}
We need to show that $\nu_I^{-1}([0,n])=\{M\in X\mid w_M(I)\leq n\}$ is closed in $X$ for every $n\inN$. We first note that $\nu_I^{-1}(0)=\D(I)\cap X$ is closed in the inverse topology since $I$ is finitely generated; hence, we only need to consider $X_n:=\nu_I^{-1}([1,n])$.

Let $J:=\rad(I)$ and let $L:=I^{-1}J^{n+1}$. Take a maximal ideal $M\in X$: if $M\notin \V(J)$, then $w_M(L)=0$ since $w_M(I)=0=w_M(J)$. On the other hand, if $M\in \V(J)$, then $w_M(L)=(n+1)w_M(J)-w_M(I)=(n+1)-w_M(I)$. Hence, $w_M(L)>0$ if and only if $w_M(I)<n+1$, i.e., if and only if $M\in X_n$. It follows that $\V(L\cap D)=X_n$ is closed in the Zariski topology. However, the Zariski, inverse and constructible topology agree on $\V(I)$, and thus (since $X_n\subseteq\V(I)$ and $\V(I)$ is closed in $X$) it follows that $X_n$ is also closed in the inverse topology. Thus, $\nu_I^{-1}([0,n])=\D(I)\cup X_n$ is closed in the inverse topology.
\end{proof}

In general, however, $\nu_I$ is not continuous, not even for finitely generated ideals; indeed, it may even not be bounded. We say that $I$ is \emph{bounded} (with respect to $X$) if $\nu_I$ is bounded.

To understand what happens when the ideal functions are not continuous, we now reason as in the case of almost Dedekind domains.
\begin{defin}
Let $D$ be a Pr\"ufer domain, let $X\subseteq\Max(D)$ be a discrete splitting set of $D$ and let $M\in X$. We say that $M$ is:
\begin{itemize}
\item \emph{critical with respect to $X$} if there is no radical ideal $I\in\Inv_X(D)$ such that $I\subseteq M$;
\item \emph{bounded-critical with respect to $X$} if there is no bounded ideal $I\in\Inv_X(D)$ such that $I\subseteq M$.
\end{itemize} 
We denote by $\inscrit_X(D)$ and $\bcrit_X(D)$, respectively, the set of critical and bounded-critical ideals of $D$ with respect to $X$.
\end{defin}

We characterize what it means for an ideal to not meet $\inscrit_X(D)$ and $\bcrit_X(D)$. Compare the following two propositions with the analogous characterizations for almost Dedekind domains (\cite[Proposition 3.10]{SP-scattered} and \cite[Propostion 3.4]{bounded-almded}).
\begin{prop}\label{prop:caratt-VI-crit}
Let $D$ be a Pr\"ufer domain, let $X\subseteq\Max(D)$ be a discrete splitting set of $D$, and $I$ be a proper ideal in $\Inv_X(D)$. The following are equivalent:
\begin{enumerate}[(i)]
\item\label{prop:caratt-VI-crit:empty} $\V(I)\cap\inscrit_X(D)=\emptyset$;
\item\label{prop:caratt-VI-crit:rad} $\rad(I)$ is finitely generated;
\item\label{prop:caratt-VI-crit:fact} there are finitely generated radical ideals $J_1,\ldots,J_n$ such that $I=J_1\cdots J_n$;
\item\label{prop:caratt-VI-crit:cont} $\nu_I$ is continuous.
\end{enumerate}
\end{prop}
\begin{proof}
\ref{prop:caratt-VI-crit:empty} $\Longrightarrow$ \ref{prop:caratt-VI-crit:rad} Suppose that $\V(I)\cap\inscrit_X(D)=\emptyset$. For every $M\in\V(I)$, let $J_M\in\Inv_X(D)$ be a radical ideal contained in $M$; then, $\{\V(J_M+I)\mid M\in\V(I)\}$ is an open cover of $\V(I)$ (with respect to the inverse topology), and since $\V(I)$ is compact in the inverse topology there is a finite subcover $\V(J_1+I),\ldots\V(J_n+I)$. Let $J:=\bigcap_i(J_i+I)$. Then, $J$ is radical since each $J_i+I$ is radical, it contains $I$, and is contained in any $M\in\V(I)$; hence $J=\rad(I)$. Moreover, $J$ is finitely generated since each $J_i+I$ is finitely generated and $D$ is a Pr\"ufer domain, and its support is contained in $X$ since the same happens for each $J_i+I$. The claim is proved.

\ref{prop:caratt-VI-crit:rad} $\Longrightarrow$ \ref{prop:caratt-VI-crit:fact} We first note that, if $\rad(I)$ is finitely generated, then there is an integer $n$ such that $\rad(I)^n\subseteq I$; in particular, $\nu_I$ is bounded. Thus we can proceed by induction on $k:=\sup\nu_I$.

If $k=1$ then $I$ is radical and $I=I$ is a factorization. Suppose that the claim is true up to $k-1$. By hypothesis, $\rad(I)$ is a finitely generated ideal, and $\nu_{\rad(I)}=\chi_{\V(I)}$. Hence, setting $I_1:=I\cdot(\rad(I))^{-1}$, we have
\begin{equation*}
\nu_{I_1}(M)=\nu_I(M)-\nu_{\rad(I)}(M)=\nu_I(M)-1
\end{equation*}
for all $M\in\V(I)$. Hence, $\sup_{\nu_{I_1}}=k-1$, and by inductive hypothesis there are radical ideals $J_2,\ldots,J_n$ such that $I_1=J_2\cdots J_n$. Hence, $I=\rad(I)\cdot J_2\cdots J_n$ is a factorization into radical ideals. Note also that each $J_i$ is invertible (and thus finitely generated) since $I$ is invertible.

\ref{prop:caratt-VI-crit:fact} $\Longrightarrow$ \ref{prop:caratt-VI-crit:empty} Let $M\in\V(I)$. Since $I=J_1\cdots J_n\subseteq M$, there must be an $i$ such that $J_i\subseteq M$; hence, $J_i$ is a finitely generated radical ideal inside $M$, and $J_i\in\Inv_X(D)$ since $I\subseteq J_i$. Thus $M\notin\inscrit(D)$ and $\V(I)\cap\inscrit_X(D)=\emptyset$.

\ref{prop:caratt-VI-crit:fact} $\Longrightarrow$ \ref{prop:caratt-VI-crit:cont}  If $I=J_1\cdots J_n$, we have $\nu_I=\nu_{J_1}+\cdots+\nu_{J_n}$. Since $J_i$ is radical, we have $\nu_{J_i}=\chi_{\V(J_i)}$; since $J_i$ is finitely generated, $\V(J_i)$ is clopen in the inverse topology, and thus its characteristic function is continuous. Hence $\nu_I$ is also continuous. 

\ref{prop:caratt-VI-crit:cont} $\Longrightarrow$ \ref{prop:caratt-VI-crit:fact}  If $\nu_I$ is continuous, consider $X_t:=\nu_I^{-1}([t,+\infty))$ for $t>0$. Each $X_t$ is clopen in the inverse topology, and since $X_t\subseteq\V(I)$ it is also closed in the Zariski topology; hence $X_t=\V(J_t)$ for some ideal $J_t$.

Let $s$ the maximum of $\nu_I$ (note that $\nu_I$ has compact support and thus it is bounded), and let $I':=J_1\cdots J_s$. Then, $\nu_{I'}=\nu_{J_1}+\cdots+\nu_{J_s}$, and thus $\nu_{I'}(P)$ is equal to the maximum $t$ such that $P\in X_t$, i.e., $\nu_{I'}(P)=\nu_I(P)$; it follows that $I=I'$. Hence $I=J_1\cdots J_s$ is a product of radical ideals, and each $J_i$ is invertible (hence finitely generated) since $I$ is invertible.
\end{proof}

\begin{cor}
Let $D$ be a Pr\"ufer domain and let $X\subseteq\Max(D)$ be a discrete splitting set of $D$. Then, the following are equivalent:
\begin{enumerate}[(i)]
\item $\inscrit_X(D)=\emptyset$;
\item $\rad(I)$ is finitely generated for all proper ideals $I\in\Inv_X(D)$;
\item every proper ideal $I\in\Inv_X(D)$ can be factored into a finite product of finitely generated radical ideals;
\item the ideal function of every $I\in\Inv_X(D)$ is continuous.
\end{enumerate}
\end{cor}

The last point of the previous corollary can be rephrased by saying that, under the assumption $\inscrit_X(D)=\emptyset$, the map $\Inv_X(D)\longrightarrow\insfunct(X,\insZ)$ can be actually thought of as a map $\Inv_X(D)\longrightarrow\inscont(X,\insZ)$. In this case, we can actually determine the range of this map; we shall do it in a greater generality in the following section (see Lemma \ref{lemma:csd-continuous} and Proposition \ref{prop:freekernel}).

We now turn to bounded ideals and bounded-critical maximal ideals.
\begin{prop}\label{prop:caratt-VI-bound}
Let $D$ be a Pr\"ufer domain, let $X\subseteq\Max(D)$ be a discrete splitting set of $D$, and let $I\in\Inv_X(D)$. Then, $I$ is bounded if and only if $\V(I)\cap\bcrit_X(D)=\emptyset$.
\end{prop}
\begin{proof}
If $I$ is bounded, then no maximal ideals containing $I$ is bounded-critical, and thus $\V(I)\cap\bcrit_X(D)=\emptyset$.

Conversely, suppose that $\V(I)\cap\bcrit_X(D)=\emptyset$. For each $M\in \V(I)$, let $J_M\in\Inv_X(D)$ be a bounded finitely generated ideal such that $J_M\subseteq M$. (These ideals $J_M$ exist since no $M$ is bounded-critical.) Then, $L_M:=I+J_M$ is a bounded finitely generated ideal such that $I\subseteq L_M\subseteq M$. Therefore, $\{\V(L_M)\mid M\in \V(I)\}$ is an open cover of $\V(I)$ (with respect to the inverse topology). The space $\V(I)$ is compact with respect to the inverse topology, and thus we can find a finite subcover $\{\V(L_{M_1}),\ldots,\V(L_{M_k})\}$. Let $L:=L_{M_1}\cdots L_{M_k}$. Then, $L$ is finitely generated and $\V(L)=\V(I)$, i.e. $\rad(L)=\rad(I)$; therefore, there is a $t$ such that $L^t\subseteq I$. However,
\begin{equation*}
\nu_{L^t}=t\nu_L=t(\nu_{L_{M_1}}+\cdots+\nu_{L_{M_k}}),
\end{equation*}
and thus $\nu_{L^t}$ is bounded; hence $\nu_I\leq\nu_{L^t}$ is bounded too. The claim is proved.
\end{proof}

\begin{cor}\label{cor:SP-bounded}
Let $D$ be a Pr\"ufer domain and let $X\subseteq\Max(D)$ be a discrete splitting set of $D$. Then, the following are equivalent:
\begin{enumerate}[(i)]
\item\label{cor:SP-bounded:critx} $\bcrit_X(D)=\emptyset$;
\item\label{cor:SP-bounded:prop} every proper ideal $I\in\Inv_X(D)$ is bounded;
\item\label{cor:SP-bounded:inv} every $I\in\Inv_X(D)$ is bounded.
\end{enumerate}
\end{cor}
\begin{proof}
\ref{cor:SP-bounded:critx} $\iff$ \ref{cor:SP-bounded:inv} follows from Proposition \ref{prop:caratt-VI-bound}. \ref{cor:SP-bounded:inv} $\Longrightarrow$ \ref{cor:SP-bounded:prop} is obvious. To show \ref{cor:SP-bounded:prop} $\Longrightarrow$ \ref{cor:SP-bounded:inv}, let $I\in\Inv_X(D)$. Then, $I=JL^{-1}$ for some proper ideals $J,L$ of $D$ with $\supp(I)\subseteq X$ (Lemma \ref{lemma:JL-1}), so that $J,L\in\Inv_X(D)$ and thus $\nu_I=\nu_J-\nu_L$ is bounded.
\end{proof}

\begin{oss}
As done in the almost Dedekind case \cite[Definitions 3.1 and 3.2]{bounded-almded} it is also possible to define the set of \emph{$n$-critical ideals} $\phantom{}^n\inscrit_X(D)$ as the set of all $M\in X$ such that no $I\in\Inv_X(D)$ with $I\subseteq M$ satisfies $\sup\nu_I\leq n$. However, this level of detail is not needed for our study.
\end{oss}

To conclude this section, we show that every discrete splitting set contains ideals that are not critical. We again follow the lead of the almost Dedekind domain case (cfr Lemma 4.4 and Theorem 4.5 of \cite{bounded-almded}).
\begin{defin}
If $X\subseteq\Max(D)$, the \emph{$X$-Jacobson radical} of $D$ is
\begin{equation*}
\Jac_X(D):=\bigcap\{P\mid P\in X\}.
\end{equation*}
\end{defin}

\begin{lemma}\label{lemma:dense}
Let $D$ be a Pr\"ufer domain and let $X\subseteq\Max(D)$ be a discrete splitting set of $D$. Suppose that $\V(\Jac_X(D))=X$ and $\bcrit_X(D)=X$. If $I$ is a proper ideal in $\Inv_X(D)$, then for every $n\inN$ the set 
\begin{equation*}
Y_n:=\nu_I^{-1}((n,+\infty))=\{M\in X\mid v_M(I)>n\}
\end{equation*}
is dense in $X$, with respect to the inverse topology.
\end{lemma}
\begin{proof}
Suppose that $Y_n$ is not dense: then, there is a finitely generated ideal $J$ such that the clopen set $\V(J)$ does not meet the closure of $Y_n$. Let $J':=J+I$: then, $J'\neq D$ since every maximal ideal containing $J$ contains also $J'$ (since $I\subseteq\Jac_X(D)$), and $\supp(J')\subseteq X$ since $J$ contains $I$. Let $M\in X$: then, $w_M(J')=0$ if $M\in Y_n$ (since $M\notin\V(J)$), while $w_M(J')\leq w_M(I)\leq n$ if $M\notin Y_n$. Hence, $\nu_{J'}$ is bounded by $n$, against the fact that $\bcrit_X(D)=X$. Thus $Y_n$ is dense in $X$.
\end{proof}

\begin{teor}\label{teor:exist-noncrit}
Let $D$ be a Pr\"ufer domain and let $X\subseteq\Max(D)$ be a discrete splitting set of $D$. Then, $\inscrit_X(D)\neq X$.
\end{teor}
\begin{proof}
We start by showing that $\bcrit_X(D)\neq X$. Suppose on the other hand that $\bcrit_X(D)= X$, and suppose first that $\Jac_X(D)$ contains an ideal $I\in\Inv_X(D)$: then, $\nu_I$ is not bounded. Let $Y_n:=\nu_I^{-1}((n,+\infty))=\{M\in\Max(D)\mid w_M(I)>n\}$. By Proposition \ref{prop:semicont}, each $Y_n$ is open in $X$, and by Lemma \ref{lemma:dense} they are all dense in $X$. Therefore, $\{Y_n\}_{n\inN}$ is a family of open and dense subsets of $X$; since $X$ is Hausdorff in the inverse topology (since the inverse and the constructible topology agree on $\Max(D)$ \cite[Corollary 4.4.9]{spectralspaces-libro}), by the Baire category theorem its intersection must be dense, and in particular nonempty. However, if $M\in\bigcap_nY_n$, then $w_M(I)>n$ for every $n$; since this cannot happen, it follows that we cannot have $\bcrit_X(D)=X$.

Consider now the general case, and let $I\in\Inv_X(D)$. Let $T:=(1+I)^{-1}D=\bigcap\{D_M\mid M\in\V(I)\}$: then, $T$ is an overring of $D$ whose maximal ideals are exactly the extensions of the prime ideals containing $I$. Hence, $X':=\{PT\mid P\in\V(I)\}=\V(IT)=\Max(T)$ is a splitting set of $T$ (since $IT$ is finitely generated over $T$) and $\Jac_{X'}(D)$ contains $I$, so that in particular it contains finitely generated ideals with support inside $X$. By the previous part of the proof, there is a proper finitely generated ideal $J$ of $T$ such that $\nu_J:X'\longrightarrow\insN$ is bounded. Let $J_0$ be a finitely generated ideal of $D$ such that $J=J_0T$, and let $L:=J_0+I$. Then, $\nu_L(M)=w_M(L)=0$ if $M\notin\V(I)$, while $\nu_L(M)=w_M(L)\leq w_{MT}(J)=\nu_J(MT)$ if $M\in\V(I)$. Furthermore, $L\neq D$ since $LT$ is contained in every maximal ideal of $T$ containing $J$. Therefore, $\nu_L$ is a nonzero bounded function, and so $\bcrit_X(D)\neq X$.

Since $\bcrit_X(D)\neq X$, there is a bounded finitely generated ideal $I\in\Inv_X(D)$. Let $J:=\rad(I)$; then, $J^n\subseteq I$ for some $n$ (since $I$ is bounded) and thus
\begin{equation*}
I=I^v\supseteq (J^n)^v=(J\cdots J)^v=(J^v\cdots J^v)^v=((J^v)^n)^v.
\end{equation*}
In particular, $J^v\neq D$, and so there is an $\alpha\in K$ such that $J\subseteq D\cap\alpha D\subsetneq D$. Hence, $L:=D\cap\alpha D$ is a radical finitely generated ideal containing $J$ (thus with support inside $X$) and in particular every maximal ideal containing $M$ is not critical. Hence $\inscrit_X(D)\neq X$.
\end{proof}

\section{The critical sequence}
Let $D$ be a Pr\"ufer domain and let $X\subseteq\Max(D)$ be a discrete splitting set of $D$. To simplify the notation, given an overring $T$ of $D$, we set $XT:=\{PT\mid P\in X,~PT\neq T\}$.

As for the case of almost Dedekind domains (see \cite[Section 5]{SP-scattered}), it is possible to recursively define a sequence $\{\inscrit_X^\alpha(D)\}_\alpha$ of subsets of $X$ and a sequence $\{T_\alpha\}_\alpha$ of overrings in the following way:
\begin{itemize}
\item $\inscrit_X^0(D):=X$;
\item if $\alpha=\gamma+1$ is a successor ordinal, then
\begin{equation*}
\inscrit_X^\alpha(D):=\{P\in X\mid PT_\gamma\in\inscrit_{XT_\gamma}(T_\gamma)\};
\end{equation*}
\item if $\alpha$ is a limit ordinal, then
\begin{equation*}
\inscrit_X^\alpha(D):=\bigcap_{\gamma<\alpha}\inscrit_X^\gamma(D);
\end{equation*}
\item for every ordinal $\alpha$,
\begin{equation*}
T_\alpha:=\bigcap\{D_P\mid P\in\inscrit_X^\alpha(D)\cup(\Spec(D)\setminus X)\}=\Omega(X\setminus\inscrit_X^\alpha(D)).
\end{equation*}
\end{itemize} 

We summarize the main properties of this sequence in the following proposition.
\begin{prop}\label{prop:Talpha}
Preserve the notation above. Then, the following hold.
\begin{enumerate}[(a)]
\item\label{prop:Talpha:closed} $\{\inscrit_X^\alpha(D)\}_\alpha$ is a descending sequence of closed subsets of $X$, with respect to the inverse topology.
\item\label{prop:Talpha:overrings} $\{T_\alpha\}_\alpha$ is an ascending sequence of overrings of $D$.
\item\label{prop:Talpha:ext1} If $\alpha=\beta+1$, then $\Max(T_\alpha)\cap XT_\alpha=\{PT_\alpha\mid P\in\inscrit_{XT_\beta}(T_\beta)\}$.
\item\label{prop:Talpha:ext2} For every $\alpha$, we have
\begin{equation*}
\Max(T_\alpha)\cap XT_\alpha=\{PT_\alpha\mid P\in\inscrit_X^\alpha(D)\}.
\end{equation*}
\end{enumerate}
\end{prop}
\begin{proof}
\ref{prop:Talpha:closed} Clearly the sequence is descending; we show that its members are closed by induction on $\alpha$. If $\alpha=0$ the claim is trivial, as $\inscrit_X^0(D)=X$. Suppose that $\alpha=\gamma+1$ is a successor ordinal, and let $M\in X\setminus\inscrit_X^\alpha(D)$: then, $MT_\gamma$ is not critical in $T_\gamma$, and thus there is a radical ideal $I\in\Inv_{XT_\gamma}(T_\gamma)$ such that $I\subseteq MT_\gamma$. Let $J\in\Inv_X(D)$ be a finitely generated ideal of $D$ such that $JT_\gamma=I$: then, no prime ideal in  $XT_\gamma$ is critical, and thus $\V(J)\cap\inscrit_X^\alpha(D)$ is empty. Thus $X\setminus\inscrit_X^\alpha(D)$ is open with respect to the inverse topology, i.e., $\inscrit_X^\alpha(D)$ is closed.

If now $\alpha$ is a limit ordinal, then by definition $\inscrit_X^\alpha(D)=\bigcap_{\beta<\alpha}\inscrit^\beta_X(D)$, and so by induction $\inscrit_X^\alpha(D)$ is an intersection of closed subsets. Hence it is closed itself.

\ref{prop:Talpha:overrings} follows from the definition of $T_\alpha$ and the fact that $\{\inscrit_X^\alpha(D)\}_\alpha$ is descending. \ref{prop:Talpha:ext1} and \ref{prop:Talpha:ext2} follow from the representation $T_\alpha=\Omega(X\setminus\inscrit_X^\alpha(D))$.
\end{proof}

\begin{teor}\label{teor:SP-scat}
Preserve the notation above. There is an ordinal number $\beta$ such that $\inscrit_X^\beta(D)=\emptyset$.
\end{teor}
\begin{proof}
Since $\{\inscrit_X^\alpha(D)\}_\alpha$ is a descending chain, there must be a $\beta$ such that $\inscrit_X^\beta(D)=\inscrit_X^{\beta+1}(D)$, and so $T_\beta=T_{\beta+1}$. If $\inscrit^\beta_X(D)\neq\emptyset$, by Proposition \ref{prop:Talpha} it follows that $\inscrit_{XT_\beta}(T_\beta)=XT_\beta$, contradicting Theorem \ref{teor:exist-noncrit}. Thus $XT_\beta=\inscrit_X^\beta(D)$ must be empty.
\end{proof}

\begin{defin}
We call the minimal ordinal number $\alpha$ such that $\inscrit_X^\alpha(D)=\emptyset$ the \emph{SP-rank} of $D$ with respect to $X$.
\end{defin}

The results we obtain for this sequence mirror the ones obtained for the case of almost Dedekind domains.
\begin{lemma}\label{lemma:csd-continuous}
Let $D$ be a Pr\"ufer domain and let $X\subseteq\Max(D)$ be a discrete splitting set of $D$.  If $f:X\longrightarrow\insZ$ is a continuous function such that $\supp(f)$ is compact and disjoint from $\inscrit(D)$, then $f=\nu_I$ for some $I\in\Inv_X(D)$.
\end{lemma}
\begin{proof}
Since $f:X\longrightarrow\insZ$ has compact support, we can write $f=\sum_ia_i\chi_{X_i}$ for some clopen subsets $X_i\subseteq X\setminus\inscrit_X(D)$. Then, $X_i$ is closed in the Zariski topology and open in the inverse topology, hence $X_i=\V(J'_i)$ for some finitely generated ideal $J'_i$. Since $X_i\cap\inscrit_X(D)=\emptyset$, by Proposition \ref{prop:caratt-VI-crit} $J_i:=\rad(J'_i)$ is finitely generated. Then, $\chi_{X_i}=\nu_{J_i}$ and $f$ is the ideal function of $\prod_iJ_i^{a_i}\in\Inv_X(D)$.
\end{proof}

\begin{prop}\label{prop:freekernel}
Let $D$ be a Pr\"ufer domain and let $X\subseteq\Max(D)$ be a discrete splitting set of $D$. Let $\Psi_\gamma:\Inv_{XT_\gamma}(T_\gamma)\longrightarrow\Inv_{XT_{\gamma+1}}(T_{\gamma+1})$ be the extension map. Then,
\begin{equation*}
\ker\Psi_\gamma\simeq\inscont_c(\inscrit_X^\gamma(D)\setminus\inscrit_X^{\gamma+1}(D),\insZ).
\end{equation*}
\end{prop}
\begin{proof}
By definition, $T_{\gamma+1}$ is just the first element of the sequence of overrings if we use $T_\gamma$ instead of $D$ as a starting point (and $XT_\gamma$ as the discrete splitting set). Hence it is enough to prove the claim for $\gamma=0$, i.e., for the extension $\Psi:\Inv(D)\longrightarrow\Inv(T_1)$. 

The kernel of $\Psi$ is the set of invertible ideals $I\in\Inv_X(D)$ such that $ID_M=D_M$ for every $M\in\inscrit_X(D)$, i.e., such that $\supp(I)\cap\inscrit_X(D)=\emptyset$, or $\supp(I)\subseteq X\setminus\inscrit(D)$. By Lemma \ref{lemma:csd-continuous}, the ideal functions of such ideals are exactly the continuous functions $X\setminus\inscrit(D)\longrightarrow\insZ$ with compact support, and thus $\ker\Psi\simeq\inscont_c(X\setminus\inscrit(D),\insZ)$. The claim is proved.
\end{proof}

\begin{cor}\label{cor:invX-sp}
Let $D$ be a Pr\"ufer domain and let $X\subseteq\Max(D)$ be a discrete splitting set of $D$. If $\inscrit_X(D)=\emptyset$, then $\Inv_X(D)\simeq\inscont_c(X,\insZ)$.
\end{cor}
\begin{proof}
If $\inscrit_X(D)=\emptyset$, then $\inscrit^0_X(D)\setminus\inscrit^1_X(D)=X$ and $T_1=\Omega(X)$, so that $XT_1=\emptyset$. The claim follows from Proposition \ref{prop:freekernel}.
\end{proof}

These results set up an inductive reasoning; we recall explicitly a lemma which we will also use later.
\begin{lemma}\label{lemma:unionesgr}
Let $G$ be an abelian group. Let $\{G_\lambda\}_{\lambda\in\Lambda}$ be a well-ordered set of ascending chain of subgroups of $G$ such that:
\begin{itemize}
\item each $G_\lambda$ is free;
\item for all $\lambda$, $G_{\lambda+1}\simeq G_\lambda\oplus H_\lambda$ for some subgroup $H_\lambda$;
\item if $\lambda$ is a limit ordinal, then $G_\lambda=\bigcup_{\alpha<\lambda}G_\alpha$;
\item $G=\bigcup_\lambda G_\lambda$.
\end{itemize}
Then, $\displaystyle{G\simeq G_0\oplus\bigoplus_{\lambda\in\Lambda} H_\lambda}$, and in particular it is free.
\end{lemma}
\begin{proof}
The statement is equivalent to \cite[Chapter 3, Lemma 7.4]{fuchs-abeliangroups}.
\end{proof}

\begin{teor}\label{teor:Invd}
Let $D$ be a Pr\"ufer domain and let $X\subseteq\Max(D)$ be a discrete splitting set of $D$. For every ordinal number $\beta$, let $Y_\beta:=\inscrit_X^\beta(D)\setminus\inscrit_X^{\beta+1}(D)$. Then, the following hold.
\begin{enumerate}[(a)]
\item\label{teor:Invd:seq} For every ordinal $\gamma$, there is an exact sequence
\begin{equation*}
0\longrightarrow\bigoplus_{\beta<\gamma}\inscont_c(Y_\beta,\insZ)\longrightarrow\Inv_X(D)\longrightarrow\Inv_X(T_\gamma)\longrightarrow 0.
\end{equation*}
\item\label{teor:Invd:spr} If $\gamma$ is the SP-rank of $D$ with respect to $X$, then
\begin{equation*}
\Inv_X(D)\simeq\bigoplus_{\beta<\gamma}\inscont_c(Y_\beta,\insZ),
\end{equation*}
\item\label{teor:Invd:free} $\Inv_X(D)$ is free.
\end{enumerate}
\end{teor}
\begin{proof}
\ref{teor:Invd:seq} We proceed by induction on $\gamma$. If $\gamma=0$ the claim is trivial. If $\gamma=1$, then the extension map $\Inv_X(D)\longrightarrow\Inv_{XT_1}(T_1)$ is surjective and its kernel is isomorphic to $\inscont_c(Y_1,\insZ)$ by Proposition \ref{prop:freekernel}. Thus we have the claimed exact sequence.

Suppose now that the claim holds for every $\lambda<\gamma$. 

Suppose that $\gamma=\alpha+1$ is a successor ordinal. Let $K_\beta$ be the kernel of the extension map $\Inv_X(D)\longrightarrow\Inv_{XT_\beta}(T_\beta)$, and let $K_{\alpha,\gamma}$ the kernel of $\Inv_{XT_\alpha}(T_\alpha)\longrightarrow\Inv_{XT_\gamma}(T_\gamma)$. From the chain of surjective maps $\Inv_X(D)\longrightarrow\Inv_{XT_\gamma}(T_\gamma)\longrightarrow\Inv_{XT_\alpha}(T_\alpha)$, we obtain an exact sequence
\begin{equation*}
0\longrightarrow K_\alpha\longrightarrow K_\gamma\longrightarrow K_{\alpha,\gamma}\longrightarrow 0.
\end{equation*}
By Proposition \ref{prop:freekernel}, $K_{\alpha,\gamma}\simeq\inscont_c(Y_\gamma,\insZ)$ is free; hence, the previous exact sequence splits and $K_\gamma\simeq K_\alpha\oplus\inscont_c(Y_\gamma,\insZ)$. The claim follows by induction.

Suppose that $\gamma$ is a limit ordinal. The sequence $\{K_\beta\}_{\beta<\gamma}$ is an ascending sequence of free groups, and by the previous part of the proof each $K_\beta$ is a direct summand of $K_{\beta+1}$. We claim that $K_\gamma$ is the union of all the $K_\beta$. Indeed, suppose that $I\in K_\gamma$: then, $IT_\gamma=T_\gamma$, and thus there is a $\beta<\gamma$ such that $IT_\beta=T_\beta$. In particular, $I\in K_\beta$. By Lemma \ref{lemma:unionesgr}, it follows that $K_\gamma$ is free as well and has the desired decomposition.

\ref{teor:Invd:spr} If $\gamma$ is the SP-rank of $D$, then $XT_\gamma=\emptyset$, and thus $\Inv_{XT_\gamma}(T_\gamma)$ is the trivial group. The claim follows from the previous part of the proof.

\ref{teor:Invd:free} Each $\inscont_c(Y_\beta,\insZ)$ is free since it is a subgroup of the free group $\insfunct_b(Y_\beta,\insZ)$. Therefore, the direct sum $\Inv_X(D)\simeq\bigoplus_\beta\inscont_c(Y_\beta,\insZ)$ is free as well.
\end{proof}

To conclude this section, we improve Theorem \ref{teor:exist-noncrit} by showing that the set $X\setminus\inscrit_X(D)$ is not only nonempty, but also dense in $X$; we need the following definition.

\begin{defin}
Let $D$ be a Pr\"ufer domain and let $X\subseteq\Max(D)$ be a discrete splitting set of $D$. The \emph{SP-height} of $P$ with respect to $X$ is the ordinal $\alpha$ such that $P\in\inscrit_X^\alpha(D)\setminus\inscrit_X^{\alpha+1}(D)$; equivalently, the SP-height of $P$ is the maximal ordinal number $\alpha$ such that $P\in\inscrit_X^\alpha(D)$.
\end{defin}
Note that this definition is well-posed, since $\bigcup_\alpha X\setminus\inscrit_X^\alpha(D)=\emptyset$, by Theorem \ref{teor:SP-scat}.

\begin{teor}\label{teor:dense}
Let $D$ be a Pr\"ufer domain and let $X\subseteq\Max(D)$ be a discrete splitting set of $D$. Then, the set $X\setminus\inscrit_X(D)$ is dense in $X$.
\end{teor}
\begin{proof}
Set $Z:=X\setminus\inscrit_X(D)$, and let $P\in X$: we show that $P\in\cl_\inverse(Z)$ by induction on the SP-height $\alpha$ of $P$ (with respect to $X$).

If the height is $0$, then $P\in Z\subseteq\cl_\inverse(Z)$. Suppose that $\alpha>0$ (so $P$ is critical) and that, for every prime ideal $Q$ of SP-height $\beta<\alpha$, we have $Q\in\cl_\inverse(Z)$; then, $\cl_\inverse(Z)$ contains $Z_\alpha:=X\setminus\inscrit_X^\alpha(D)$. By definition, $PT_\alpha$ is not critical in $T_\alpha$, and thus there is a finitely generated ideal $I$ of $D$ such that $\supp(I)\subseteq X$, $IT_\alpha$ is radical in $T_\alpha$ and $IT_\alpha\subseteq PT_\alpha$ (thus, $I\subseteq P$). If $P\notin\cl_\inverse(Z)$, there is an open subset $\Omega$ of $X$ containing $P$ but disjoint from $\cl_\inverse(Z)$ and thus from $Z_\alpha$; since the $\V(J)$ (as $J$ ranges among the finitely generated ideals of $D$) are a subbasis of $X$, with respect to the inverse topology, we can suppose that $\Omega=\V(J)$. Let $I':=I+J$: then, $I'$ is finitely generated, $\supp(I')\subseteq X$, $I'\subseteq P$, and $\V(I')\cap Z_\alpha\subseteq\V(J)\cap Z_\alpha=\emptyset$. Moreover, if $M\notin Z_\alpha$, then $\nu_{I'}(M)\leq\nu_I(M)=\nu_{IT_\alpha}(M)\leq 1$; hence, $I'$ is a radical ideal. This would imply that $P$ is a non-critical ideal of $D$ with respect to $X$, a contradiction; hence $P$ must belong to $\cl_\inverse(Z)$. By induction, $\cl_\inverse(Z)=X$, i.e., $Z$ is dense in $X$.
\end{proof}

\section{Divisorial ideals (continued)}\label{sect:DivX-2}
We record in this brief section two consequence for the group $\Div_X(D)$ of the analysis done for invertible ideals. Both are generalizations of analogous results for almost Dedekind domains (see \cite[Theorem 6.5]{SP-scattered} and \cite[Proposition 3.9]{bounded-almded}).
\begin{prop}
Let $D$ be a Pr\"ufer domain and let $X\subseteq\Max(D)$ be a discrete splitting set. If $\bcrit_X(D)=\emptyset$, then $\Div_X(D)$ is free.
\end{prop}
\begin{proof}
If $\bcrit_X(D)$ is empty, then every $I\in\Inv_X(D)$ is bounded (Corollary \ref{cor:SP-bounded}), and thus we have an injective group homomorphism $\Psi:\Inv_X(D)\longrightarrow\insbound(X,\insZ)$. The latter is a complete $\ell$-group, and thus the completion of $\Inv_X(D)$ is isomorphic to a subgroup of $\insbound(X,\insZ)$; since the completion of $\Inv_X(D)$ is $\Div_X(D)$ (Theorem \ref{teor:DivX}), the claim is proved.
\end{proof}

Note that we can't use the canonical map $\Div_X(D)\longrightarrow\insbound(X,\insZ)$, since it is not a group homomorphism (see Remark \ref{oss:Psi}).

Recall that, for a topological space $X$, we denote by $E_X$ the Gleason cover of $X$ \cite{strauss-extremallydisconnected}.
\begin{prop}\label{prop:Div-SP}
Let $D$ be a Pr\"ufer domain and let $X\subseteq\Max(D)$ be a discrete splitting set. If $\inscrit_X(D)=\emptyset$, then there is a commutative diagram
\begin{equation*}
\begin{CD}
\Inv_X(D)  @>{\subseteq}>> \Div_X(D)\\
@VVV    @VVV\\
\funcontcomp(X,\insZ) @>>> \funcontcomp(E_X,\insZ)
\end{CD}
\end{equation*}
such that the vertical arrows are isomorphisms. In particular, $\Div_X(D)$ and $\Div_X(D)/\Inv_X(D)$ are free groups.
\end{prop}
\begin{proof}
If $\inscrit_X(D)=\emptyset$, then $\Inv_X(D)\simeq\inscont_c(X,\insZ)$ (Corollary \ref{cor:invX-sp}); thus, the completion of $\Inv_X(D)$ is isomorphic to $\inscont_c(E_X,\insZ)$. Since $\Div_X(D)$ is the completion of $\Inv_X(D)$ (Teorem \ref{teor:DivX}), the diagram is well-defined and commutative.

Hence $\Div_X(D)$ is free (as $\inscont_c(E_X,\insZ)\subseteq\insbound(E_X,\insZ)$), and by \cite[Lemma 6.4]{SP-scattered} also the quotient $\Div_X(D)/\Inv_X(D)\simeq\inscont_c(E_X,\insZ)/\inscont_c(X,\insZ)$ is free.
\end{proof}

We shall give a more topological proof of these results in \cite{residual}.

\section{Compactly layered spaces}
Proposition \ref{prop:kerext} shows that quotienting $\Inv(D)$ by $\Inv_X(D)$ is equivalent to passing from $D$ to the overring $\Omega(X)$; if $X$ is  a splitting set, this means that to study the freeness of $\Inv(D)$ we can ``throw out'' the prime ideals of $X$ and concentrate on $\Spec(D)\setminus X$, just as in the study of $\Inv_X(D)$ we focus only on $X$. For example, if $X=\Max(D)$ is a splitting set, the passage from $D$ to $\Omega(X)$ is equivalent to ``peeling off'' the maximal ideals from the spectrum, and thus it is natural to try to repeat the process, passing from $\Omega(X)$ to an even larger overring.

In order to formalize this reasoning, we introduce the following definition.
\begin{defin}
Let $R$ be a ring. We say that $\Spec(R)$ is \emph{compactly layered} if there is a well-ordered sequence $\{Y_\alpha\}_{\alpha<\lambda}$ such that:
\begin{itemize}
\item $Y_0=\Spec(R)$;
\item $Y_\gamma=\emptyset$ for some $\gamma$;
\item $Y_\alpha\subseteq Y_\beta$ if $\alpha>\beta$;
\item if $\alpha$ is a limit ordinal, $Y_\alpha=\bigcap\{Y_\beta\mid \beta<\alpha\}$;
\item each $Y_\alpha$ is closed in the inverse topology;
\item for every $\alpha$, we have $Y_\alpha\setminus Y_{\alpha+1}\subseteq\Max(Y_\alpha)$.
\end{itemize}
We call $\{Y_\alpha\}_{\alpha<\lambda}$ a \emph{sequence of layers} for $\Spec(R)$.
\end{defin}

\begin{ex}\label{ex:comp-layered}
~\begin{enumerate}
\item If $D$ is a one-dimensional domain, then $\Spec(D)$ is compactly layered, with sequence of layers $Y_0=\Spec(R)$, $Y_1=\{(0)\}$, $Y_2=\emptyset$.
\item If $\Spec(R)$ is finite, then it is compactly layered: indeed, if $n=\dim(R)$, it is enough to take $Y_i$ to be the set of prime ideals of height at most $n-i$.
\end{enumerate}
We shall give in Section \ref{sect:examples} some sufficient conditions for $\Spec(D)$ to be compactly layered.
\end{ex}

\begin{prop}\label{prop:complay}
Let $D$ be a Pr\"ufer domain such that $\Spec(D)$ is compactly layered with sequence of layers $\{Y_\alpha\}_{\alpha<\lambda}$. Let $X:=\Spec(D)\setminus Y_\beta$ for some $\beta$. Then:
\begin{enumerate}[(a)]
\item $X$ is a splitting set for $D$;
\item $\Omega(\Spec(D)\setminus Y_\beta)$ has compactly layered spectrum with sequence of layers $\{Y_\alpha\}_{\beta\leq\alpha<\lambda}$.
\end{enumerate}
\end{prop}
\begin{proof}
The first statement is equivalent to saying that $Y_\beta$ is closed in the inverse topology, which it is by definition of sequence of layers. The spectrum of $\Omega(\Spec(D)\setminus Y_\beta)$ is homeomorphic to $Y_\beta$ (by the restriction map $P\mapsto P\cap D$), and all the conditions of the statement hold.
\end{proof}

\begin{lemma}\label{lemma:unionOmega}
Let $D$ be a Pr\"ufer domain such that $\Spec(D)$ is compactly layered with sequence of layers $\{Y_\alpha\}_{\alpha<\lambda}$, and let $Z_\alpha:=\Spec(D)\setminus Y_\alpha$. If $\gamma$ is a limit ordinal, then $\Omega(Z_\gamma)=\bigcup\{\Omega(Z_\alpha)\mid \alpha<\gamma\}$.
\end{lemma}
\begin{proof}
Clearly the union is contained in $\Omega(Z_\gamma)$. Suppose $x\in\Omega(Z_\gamma)$, and let $I:=(D:_Dx)$ be its conductor. Then, $I$ is a finitely generated ideal of $D$, and $I\Omega(Z_\alpha)=\Omega(Z_\alpha)$ if and only if $x\in Z_\alpha$, if and only if $\V(I)\cap Y_\alpha=\emptyset$.

The space $\V(I)$ is an open and compact subset of $\Spec(D)$, with respect to the inverse topology. Moreover, $I\Omega(Z_\gamma)=\Omega(Z_\gamma)$, and thus $\{Z_\alpha\cap \V(I)\}_{\alpha<\gamma}$ is a family of closed subsets of $\Spec(D)$ whose intersection is empty. By the finite intersection property, $Y_\beta\cap\V(I)=\emptyset$ for some $\beta<\gamma$; then, $I\Omega(Z_\beta)=\Omega(Z_\beta)$ and $x\in\Omega(Z_\beta)$. The claim is proved.
\end{proof}

\begin{defin}
We say that a domain $D$ is a \emph{CSD-domain} if $D$ is a strongly discrete Pr\"ufer domain with compactly layered spectrum.
\end{defin}

\begin{teor}\label{teor:CSD-free}
Let $D$ be a CSD-domain. Then, $\Inv(D)$ is free.
\end{teor}
\begin{proof}
Let $\{Y_\alpha\}_{\alpha<\lambda}$ be a sequence of layers of $\Spec(D)$ and, for every $\alpha$, let $X_\alpha:=Y_\alpha\setminus Y_{\alpha+1}$ and $Z_\alpha:=\Spec(D)\setminus Y_\alpha$. We claim that, for every $\gamma>0$, there is an exact sequence
\begin{equation}\label{eq:succ-invXalpha}
0\longrightarrow\bigoplus_{\alpha<\gamma}\Inv_{X_\alpha}(\Omega(Z_\alpha))\longrightarrow\Inv(D)\longrightarrow\Inv(\Omega(Z_\gamma))\longrightarrow 0.
\end{equation}
We proceed by induction on $\gamma$; the proof is similar to the proof of Theorem \ref{teor:Invd}.

If $\gamma=1$, then $Y_0=\Spec(D)$ and $\Omega(Z_0)=D$, so that the existence of the sequence \eqref{eq:succ-invXalpha} reduces to Proposition \ref{prop:kerext}.

Suppose that the claim is true for ordinal numbers strictly smaller than $\gamma$. We distinguish two cases.

Suppose first that $\gamma=\alpha+1$ is a successor ordinal, and let $D':=\Omega(Z_\alpha)$. Then, $D'$ has compactly layered spectrum with sequence of layers $\{Y_\beta\}_{\gamma\leq\beta<\lambda}$, by Proposition \ref{prop:complay}. Consider the three extension maps $\Psi_\alpha:\Inv(D)\longrightarrow\Inv(D')$, $\Psi_\gamma:\Inv(D)\longrightarrow\Inv(\Omega(Z_\gamma))$ and $\Psi':\Inv(D')\longrightarrow\Inv(\Omega(Z_\gamma))$; let their kernels be, respectively, $K_\alpha$, $K_\gamma$ and $K_{\alpha,\gamma}$. Then, $\Psi_\gamma=\Psi'\circ\Psi_\alpha$, and thus there is an exact sequence
\begin{equation*}
0\longrightarrow K_\alpha\longrightarrow K_\gamma\longrightarrow K_{\alpha,\gamma}\longrightarrow 0.
\end{equation*}
The map $\Psi'$ is just the map involved in the base case if we start from $D'$ instead of $D$; hence, $K_{\alpha,\gamma}\simeq\Inv_{X_\alpha}(D')$, which is free by Theorem \ref{teor:Invd}. Therefore, the previous sequence splits and $K_\gamma\simeq K_\alpha\oplus \Inv_{X_\alpha}(D')$; the claim now follows by induction.

Suppose now that $\gamma$ is a limit ordinal, and let as above $K_\alpha$ be the kernel of $\Psi_\alpha:\Inv(D)\longrightarrow\Inv(\Omega(Z_\alpha))$. The sequence $\{K_\alpha\}_{\alpha<\gamma}$ is an ascending sequence of free groups, and by the previous part of the proof each $K_\alpha$ is a direct summand of $K_{\alpha+1}$. We claim that $K_\gamma$ is the union of all $K_\alpha$ for $\alpha<\gamma$. Indeed, if $I\in K_\gamma$, then $I\Omega(Z_\gamma)=\Omega(Z_\gamma)$. Since $\gamma$ is a limit ordinal, $\Omega(Z_\gamma)=\bigcup_{\alpha<\gamma}\Omega(Z_\alpha)$ (Lemma \ref{lemma:unionOmega}) and $I\Omega(Z_\beta)=\Omega(Z_\beta)$ for some $\beta<\gamma$, i.e., $I\in K_\beta$ for some $\beta$. By Lemma \ref{lemma:unionesgr}, $K_\gamma$ has the desired decomposition.

\bigskip

Since $\Spec(D)$ is compactly layered, there is a $\beta$ such that $Z_\beta=\emptyset$, i.e., $Z_\beta=\Spec(D)$ and $\Omega(Z_\beta)=K$. Thus, the exact sequence \eqref{eq:succ-invXalpha} becomes
\begin{equation*}
0\longrightarrow\bigoplus_{\alpha<\beta}\Inv_{X_\alpha}(\Omega(Z_\alpha))\longrightarrow\Inv(D)\longrightarrow\Inv(K)\longrightarrow 0.
\end{equation*}
Since $\Inv(K)$ is trivial, we have that $\displaystyle{\Inv(D)\simeq\bigoplus_{\alpha<\beta}\Inv_{X_\alpha}(\Omega(Z_\alpha))}$; in particular, $\Inv(D)$ is a free group.
\end{proof}

\section{Examples and applications}\label{sect:examples}
In this section, we collect some sufficient conditions and some explicit examples of Pr\"ufer domains with compactly layered spectrum. A rather wide condition is the following.
\begin{prop}\label{prop:class-complay}
Let $\mathcal{C}$ be a class of Pr\"ufer domains such that:
\begin{itemize}
\item $\mathcal{C}$ is closed by overrings (i.e., if $D\in\mathcal{C}$ and $T\in\Over(D)$ then $T\in\mathcal{C}$);
\item for every $D\in\mathcal{C}$, the interior of $\Max(D)$ (with respect to the inverse topology) is nonempty.
\end{itemize}
Then, $\Spec(D)$ is compactly layered for every $D\in\mathcal{C}$.
\end{prop}
\begin{proof}
We define the layers $Y_\alpha$ recursively for every ordinal $\alpha$. We first set $Y_0:=\Spec(D)$. Suppose that $Y_\beta$ is defined for every $\beta<\alpha$.

If $\alpha$ is a limit ordinal we set $Y_\alpha:=\bigcap_{\beta<\alpha}Y_\beta$.

If $\alpha=\gamma+1$ is a successor ordinal, then $Y_\gamma$ is homeomorphic to $\Spec(T)$ for some $T\in\Over(D)$. By hypothesis, the interior $Z$ of $\Max(T)$ is nonempty; thus, $Y_\gamma\setminus Z$ is a closed set of $\Spec(D)$, with respect to the inverse topology, that is properly contained in $Y_\gamma$, and $Y_\gamma\setminus Y_\alpha\subseteq\Max(Y_\gamma)$.

Thus we only need to show that the sequence $\{Y_\alpha\}$ reaches the empty set. For cardinality reasons, there must be a $\beta$ such that $Y_\beta=Y_{\beta+1}$; however, by construction, if $Y_\beta\neq\emptyset$ then $Y_{\beta+1}\subsetneq Y_\beta$, a contradiction. Hence $\Spec(D)$ is compactly layered.
\end{proof}

\begin{cor}\label{cor:noeth-complay}
Let $D$ be an Pr\"ufer domain such that $\Spec(D)$ is a Noetherian space. Then $\Spec(D)$ is compactly layered.
\end{cor}
\begin{proof}
If $\Spec(D)$ is Noetherian, then $\Spec(T)$ is Noetherian for every $T\in\Over(D)$, since $\Spec(T)$ is homeomorphic to a subset of $\Spec(D)$. Furthermore, $\Max(D)$ is open in the inverse topology since $\Spec(D)\setminus\Max(D)$ is compact and closed by generizations. By Proposition \ref{prop:class-complay}, every Pr\"ufer domain $D$ with these properties has compactly layered spectrum.
\end{proof}

Recall that a subset $X\subseteq\Spec(D)$ is \emph{locally finite} if every $x\in D$ is contained in only finitely many $P\in X$, or equivalently if every $x\in K$ becomes a unit in all but finitely many $D_P$ with $P\in X$. Compare the following corollary with \cite[Proposition 6.15]{inv-free}. 
\begin{cor}
Let $D$ be a Pr\"ufer domain such that $\Spec(D)\setminus\Max(D)$ is locally finite. Then $\Spec(D)$ is compactly layered.
\end{cor}
\begin{proof}
Let $\mathcal{C}$ be the class of Pr\"ufer domains satisfying the hypothesis of the statement. If $D\in\mathcal{C}$ and $T\in\Over(D)$, then every $P\in\Spec(T)$ that is not maximal restricts to a prime ideal $P\cap D$ of $D$ that is not maximal; hence $\{T_Q\mid Q\in\Spec(T)\setminus\Max(T)\}\subseteq\{D_P\mid P\in\Spec(T)\setminus\Max(T)\}$, and thus $\mathcal{C}$ is closed by overrings. Let now $I$ be a nonzero finitely generated ideal of $D$: then, $I$ is contained in only finitely many non-maximal ideals, say $P_1,\ldots,P_n$. Let $M$ be a maximal ideal containing $I$: by prime avoidance, there is an $x\in M\setminus(P_1\cup\cdots\cup P_n)$, and thus $J:=(I,x)$ is a finitely generated ideal such that $\V(I)\subseteq\Max(D)$. Hence, the interior of $\Max(D)$ contains $\V(J)$ and so is nonempty. By Proposition \ref{prop:class-complay}, every $D\in\mathcal{C}$ has compactly layered spectrum.
\end{proof}

Before stating a rather general result about CSD-domain, we premise a lemma that is of independent interest.
\begin{lemma}\label{lemma:minimal-jump}
Let $D$ be an integral domain, and let $I$ be a proper finitely generated ideal such that $\Min(I)$ is compact and $\V(I)\nsubseteq\Max(D)$. Let $M\in\V(I)\setminus\Min(I)$. Then, there is a finitely generated ideal $J$ such that $I\subseteq J\subseteq M$ and $\Min(I)\cap\Min(J)=\emptyset$.
\end{lemma}
\begin{proof}
Since $M\notin\Min(I)$, for each $P\in\Min(I)$ we can find an element $x_P\in M\setminus P$; thus, $\{\D(x_P)\mid P\in\Min(I)\}$ is an open cover of $\Min(I)$. Since $\Min(I)$ is compact, it has a finite subcover $\D(x_1),\ldots,\D(x_n)$. Let $J:=I+(x_1,\ldots,x_n)D$: then, $J$ is finitely generated and, by construction, $J\subseteq M$. Finally, if $P\in\Min(I)$ then $J\nsubseteq I$ since $P\in\D(x_i)$ for some $i$ and $x_i\in J$. Thus the claim is proved.
\end{proof}

\begin{prop}
Let $D$ be a strongly discrete Pr\"ufer domain such that $\Min(I)$ is compact (with respect to the Zariski topology) for every finitely generated ideal $I$. Then, $D$ is a CSD-domain.
\end{prop}
\begin{proof}
Let $\mathcal{C}$ be the class of Pr\"ufer domains satisfying the hypothesis. Then, $\mathcal{C}$ is closed by overrings: indeed, if $D\in\mathcal{C}$, $T\in\Over(D)$, and $I$ is a finitely generated ideal of $T$, then $I=JT$ for some finitely generated ideal $J$ of $D$. Let $\iota:\Spec(T)\longrightarrow\Spec(D)$ be the restriction map. Then, $\iota(\Min(I))=\Min(J)\cap\iota(\Spec(T))$, and
\begin{equation*}
\iota(\Min(I)^\downarrow)=\iota(\Min(I))^\downarrow=\Min(J)^\downarrow\cap\iota(\Spec(T)),
\end{equation*}
In particular, both $\Min(J)^\downarrow$ and $\iota(\Spec(T))$ are closed in the inverse topology (the former since $\Min(J)$ is compact), which means that also $\iota(\Min(I))^\downarrow$ is closed, and so $\iota(\Min(I))$ is compact. Since $\iota(\Min(I))$ is homeomorphic to $\Min(I)$, also $\Min(I)$ is compact, and $T\in\mathcal{C}$.

Let now $D\in\mathcal{C}$, and suppose that the interior of $\Max(D)$, with respect to the inverse topology, is empty, i.e., there is no finitely generated ideal whose minimal primes are all maximal. Let $I_0$ be a finitely generated ideal of $D$. By hypothesis, $\V(I_0)\nsubseteq\Max(D)$, and thus by Lemma \ref{lemma:minimal-jump} there is a proper finitely generated ideal $I_1\supseteq I_0$ such that $\Min(I_0)\cap\Min(I_1)=\emptyset$. Applying repeatedly the same reasoning, we can find a sequence $I_0\subsetneq I_1\subsetneq\cdots\subsetneq I_n\subsetneq\cdots$ of finitely generated ideals of $D$ such that $I_{n+1}$ is not contained in any minimal prime of $I_n$. Then, $\{\V(I_n)\}_{n\inN}$ is a decreasing sequence of closed subsets of $\Spec(D)$ (with respect to the Zariski topology) and thus, by the finite intersection property, there is an $M\in\bigcap_n\V(I_n)$. In particular, $I_n\subseteq M$ for every $n$, and thus we can find a (unique) minimal prime $P_n$ of $I_n$ contained in $M$. By construction, $P_n\subsetneq P_{n+1}$ for every $n$; thus $\{P_n\}_{n\inN}$ is a strictly increasing sequence of prime ideals. It follows that $P:=\bigcup_nP_n$ is a prime ideal that is equal to the union of the prime ideals strictly contained in itself; however, $P^2$ too contains each $P_n$, and thus $P=P^2$, against the fact that $D$ is strongly discrete.. Therefore, the sequence of the ideals $I_n$ cannot exists, and thus the interior of $\Max(D)$ is nonempty.

By Proposition \ref{prop:class-complay}, every $D\in\mathcal{C}$ is a CSD-domain.
\end{proof}

Another way to construct compactly layered spectra is by extensions.
\begin{prop}\label{prop:fiber0}
Let $D\subseteq R$ be integral domains, and let $\phi:\Spec(R)\longrightarrow\Spec(D)$, $P\mapsto P\cap D$ be the canonical restriction map. If $\Spec(D)$ is compactly layered and the fibers of $\phi$ are $0$-dimensional, then $\Spec(R)$ is compactly layered.
\end{prop}
\begin{proof}
Let $\{Y_\alpha\}_\alpha$ be a sequence of layers for $\Spec(D)$, and let $Z_\alpha:=\phi^{-1}(Y_\alpha)$ for every $\alpha$. We claim that $\{Z_\alpha\}_\alpha$ is a sequence of layers for $\Spec(R)$.

Indeed, $Z_0=\phi^{-1}(Y_0)=\phi^{-1}(\Spec(D))=\Spec(R)$, while $Z_\alpha=\emptyset$ whenever $Y_\alpha$ is empty; moreover, $\{Z_\alpha\}$ is clearly a descending chain, and if $\alpha$ is a limit ordinal then
\begin{equation*}
Z_\alpha=\phi^{-1}(Y_\alpha)=\phi^{-1}\left(\bigcap_{\beta<\alpha}Y_\beta\right)=\bigcap_{\beta<\alpha}\phi^{-1}(Y_\beta)=\bigcap_{\beta<\alpha}Z_\beta
\end{equation*}
since $\{Y_\alpha\}$ is a descending chain.

The restriction map $\phi$ is continuous with respect to the inverse topology, and thus each $Z_\alpha$ is closed since each $Y_\alpha$ is closed. If now $P\in Z_\alpha\setminus Z_{\alpha+1}$ is not maximal in $Z_\alpha$, there is a $Q\in Z_\alpha$ such that $P\subsetneq Q$, and thus $\phi(P)\subseteq\phi(Q)$; since $\phi(P)\in Y_\alpha\setminus Y_{\alpha+1}$, we must have $\phi(P)=\phi(Q)$. However, this case is impossible since $\phi^{-1}(\phi(P))$ is zero-dimensional. Hence $\{Z_\alpha\}_\alpha$ be a sequence of layers for $\Spec(R)$, as claimed.
\end{proof}

\begin{cor}
Let $D$ be an integral domain and $R$ an integral extension of $D$. If $\Spec(D)$ is compactly layered, so is $\Spec(R)$.
\end{cor}
\begin{proof}
By the incomparability property, the fibers of the restriction map of an integral extension are $0$-dimensional. Apply Proposition \ref{prop:fiber0}.
\end{proof}

\begin{cor}\label{cor:CSD-overring}
Let $D$ be a Pr\"ufer domain and $T$ an overring of $D$. Then, the following hold. 
\begin{enumerate}[(a)]
\item If $\Spec(D)$ is compactly layered, then so is $\Spec(T)$.
\item If $D$ is a CSD-domain, then so is $T$.
\end{enumerate}
\end{cor}
\begin{proof}
Under the hypothesis on $D$ and $T$, if $\phi:\Spec(T)\longrightarrow\Spec(D)$ is the restriction map then $\phi^{-1}(P)$ is a singleton, since it only contains $PD_P\cap T$. The first claim follows from Proposition \ref{prop:fiber0}. The second claim follows by joining the previous result with the fact that an overring of a strongly discrete Pr\"ufer domain is strongly discrete.
\end{proof}

\medskip

To conclude the paper, we now give a detailed look at one interesting class of examples, namely rings of integer valued polynomials. Let $D$ be an integral domain with quotient field $K$, and let $E\subseteq K$. The \emph{ring of integer-valued polynomials} on $E$ over $D$ is
\begin{equation*}
\Int(E,D):=\{f\in K[X]\mid f(E)\subseteq D\};
\end{equation*}
we set $\Int(D):=\Int(D,D)$. This construction has many interesting properties and has been studied very deeply (see \cite{intD}); indeed, the method of using $\Inv_X(D)$ that we used in this paper, can be seen as a wide generalization of the study of $\Inv(\Int(D))$ carried out by first considering unitary ideals, where an ideal $I$ of $\Int(D)$ is \emph{unitary} if $I\cap D\neq(0)$.

Let $V$ be a discrete valuation with finite residue field, let $\mm$ be its maximal ideal, $K$ its quotient field and $v$ be its valuation. Then, $\Int(V)$ is a two-dimensional Pr\"ufer domain \cite[Proposition V.1.8, Lemma VI.1.4]{intD}, and it is strongly discrete: indeed, let $P\in\Int(V)$. If $P\cap D=(0)$, then $P=q(X)K[X]\cap\Int(V)$ for some irreducible polynomial $q(X)$ of $K[X]$, and thus $\Int(V)_P=K[X]_{(q(X))}$ is a DVR \cite[Corollary V.1.2]{intD}. If $P\cap D=\mm$, then
\begin{equation*}
P=\mathfrak{M}_\alpha:=\{f\in\Int(D)\mid f(\alpha)\in\widehat{V}\}.
\end{equation*}
for some $\alpha\in\widehat{V}$, where $\widehat{V}$ is the completion of $V$ with respect to the $\mm$-adic topology \cite[Proposition V.2.2]{intD}; in this case, $\Int(D)_P$ is s strongly discrete valuation domain since it is either a DVR or a two-dimensional valuation ring with value group $\insZ\times\insZ$, and $w_P(f)=v(f(\alpha))$ \cite[Proposition VI.1.9 and previous discussion]{intD}. Thus $\Int(V)$ is strongly discrete.

Let now $\mathcal{X}$ be the set of unitary maximal ideals. Then, $\mathcal{X}$ is open in the inverse topology, since it is equal to $\V(\mm\Int(V))$; it is also easy to see that $\Omega(\mathcal{X})=K[X]$. The assignment $\alpha\mapsto\mathfrak{M}_\alpha$, from $\widehat{V}$ to $\mathcal{X}$, is not only a bijection \cite[Theorem V.2.10]{intD}, but also a homeomorphism from the inverse topology (which coincides with the Zariski topology on $\mathcal{X}$) to the $\mm$-adic topology \cite[Chapter V, Exercise 7]{intD}. Since $\mm\Int(V)$ is radical, moreover, $\inscrit_\mathcal{X}(D)=\emptyset$, and thus unitary ideals have radical factorization.

Therefore, $\Spec(\Int(V))$ admits a sequence of layers
\begin{equation*}
\begin{aligned}
Y_0 & =\Spec(\Int(V)),\\
Y_1 & =\{P\in\Spec(\Int(V))\mid P\cap D=(0)\}=\Spec(D)\setminus \mathcal{X},\\
Y_2 & =\{(0)\},\\
Y_3 & =\emptyset.
\end{aligned}
\end{equation*}
and so $\Int(V)$ is a CSD-domain.

A fractional ideal $I$ is \emph{almost unitary} if $dI$ is unitary for some $d\in V$, $d\neq 0$: this condition is equivalent to $\supp(I)\subseteq\mathcal{X}$, and thus the subgroups $\Inv^u(\Int(V))$ and $\Div^u(\Int(V))$ of (respectively) the almost unitary invertible and divisorial ideals coincides with $\Inv_\mathcal{X}(D)$ and $\Div_\mathcal{X}(D)$. Moreover, if $I$ is almost unitary, the ideal function $\nu_I$ on $\mathcal{X}$ coincides with the ``classical'' ideal function sending the maximal ideal $\mathfrak{M}_\alpha$ to $\min\{\widehat{v}(f(\alpha))\mid f\in I\}$ (where $\widehat{v}$ denotes the valuation on $\widehat{V}$).

By Theorems \ref{teor:Invd} and \ref{teor:CSD-free}, we have the following.
\begin{teor}
Let $V$ be a discrete valuation ring with finite residue field, let $K$ be its residue field and $\widehat{V}$ be its completion. Let $\mathcal{X}$ be the set of maximal unitary ideals of $\Int(V)$. Then, $\Int(V)$ is a CSD-domain; moreover,
\begin{equation*}
\Inv(\Int(V))\simeq\inscont(\widehat{V},\insZ)\oplus\Inv(K[X]).
\end{equation*}
and
\begin{equation*}
\Div_\mathcal{X}(\Int(V))\simeq\inscont(E_{\widehat{V}},\insZ).
\end{equation*}
\end{teor}

Let now $D$ be an almost Dedekind domain with finite residue fields. It is not always true that $\Int(D)$ is a Pr\"ufer domain; however, this holds if $\Int(D)$ \emph{behaves well under localization}, i.e., if $\Int(D)_P=\Int(D_P)$ for every maximal ideal $P$ of $D$ \cite[Proposition VI.1.6]{intD}; this in particular holds if $D$ is a Dedekind domain \cite[Proposition I.2.7]{intD}. To state the result in full generality, recall that a \emph{fractional subset} of a domain $D$ is a subset $F$ of its quotient field such that $dF\subseteq D$ for some $d\in D$, $d\neq 0$. Recall also that an \emph{SP-domain} is an integral domain where every ideal has radical factorization.
\begin{teor}\label{teor:intD-ad}
Let $D$ be an almost Dedekind domain with finite residue fields, let $K$ be its residue field, $F\subseteq K$ be a fractional subset and $\mathcal{X}$ the set of unitary prime ideals of $D$. Suppose that $\Int(D)$ behaves well under localization. Then, the following hold.
\begin{enumerate}[(a)]
\item $\Int(F,D)$ is a CSD-domain.
\item If $D$ is an SP-domain, then
\begin{equation*}
\Inv(\Int(F,D))\simeq\inscont(\mathcal{X},\insZ)\oplus\Inv(K[X])
\end{equation*}
and
\begin{equation*}
\Div^u(\Int(F,D))\simeq\inscont(E_\mathcal{X},\insZ).
\end{equation*}
\item If $D$ is a Dedekind domain, then
\begin{equation*}
\Inv(\Int(E,D))\simeq\left(\bigoplus_{Q\in\Max(D)}\inscont(\widehat{F_Q},\insZ)\right)\oplus\Inv(K[X])
\end{equation*}
and
\begin{equation*}
\Div^u(\Int(F,D))\simeq\bigoplus_{Q\in\Max(D)}\inscont(E_{\widehat{F_Q}},\insZ),
\end{equation*}
where $\widehat{F_Q}$ is the completion of $F$ with respect to the $QD_Q$-adic topology.
\end{enumerate}
\end{teor}
\begin{proof}
Suppose first that $F=D$. By \cite[Proposition VI.1.6]{intD}, $\Int(D)$ is a Pr\"ufer domain; let $P$ be a prime ideal. If $P\cap D=(0)$, then as above $\Int(V)_P=K[X]_{(q(X))}$ is a DVR, where $q(X)\in K[X]$ is irreducible \cite[Corollary V.1.2]{intD}; if $P\cap D=Q\neq(0)$, then $P\Int(D)_Q=P\Int(D_Q)$ is a unitary prime ideal of $\Int(D_Q)$, and thus $\Int(D)_P=\Int(D_Q)_{PD_Q}$ is strongly discrete. Hence $\Int(D)$ is strongly discrete. Moreover, $\mathcal{X}=\bigcup\V(d\cdot\Int(D))$, as $d$ ranges among the nonzero elements of $D$, and thus $\mathcal{X}$ is open, with respect to the inverse topology; it follows that 
\begin{equation*}
\begin{aligned}
Y_0 & =\Spec(\Int(D)),\\
Y_1 & =\{P\in\Spec(\Int(D))\mid P\cap D=(0)\}=\Spec(D)\setminus \mathcal{X},\\
Y_2 & =\{(0)\},\\
Y_3 & =\emptyset.
\end{aligned}
\end{equation*}
is a sequence of layers for $\Int(D)$. Thus $\Int(D)$ is a CSD-domain.

If $D$ is an SP-domain and $P\in \mathcal{X}$, then $P\cap D$ contains a radical finitely generated ideal $I$; thus, $I\cdot\Int(D)$ is a radical finitely generated ideal contained in $P$, and so $\inscrit_\mathcal{X}(D)=\emptyset$. The claim follows from Theorems \ref{teor:Invd} and \ref{teor:CSD-free}.

By the previous reasoning, $\mathcal{X}$ is always the set-theoretical union of the sets $\mathcal{X}_Q:=\{P\in \mathcal{X}\mid P\cap D=Q\}$, which is homeomorphic to the space of unitary prime ideals of $\Int(D_Q)$, i.e., to $\widehat{D_Q}$. (Note that the notation is not ambiguous, as $D$ is dense in the completion of $D_Q$, with respect to the $QD_Q$-adic topology.) If $D$ is a Dedekind domain, moreover, each $\mathcal{X}_Q$ is clopen in the inverse topology, since is is equal to $\V(Q\Int(D))$ and $Q$ is finitely generated; therefore, $\mathcal{X}$ is homeomorphic to the disjoint union of the spaces $\widehat{D_Q}$, and so
\begin{equation*}
\inscont(\mathcal{X},\insZ)\simeq\bigoplus_{Q\in\Max(D)}\inscont(\widehat{D_Q},\insZ).
\end{equation*}
The claim about $\Inv(D)$ now follows from the previous point, while the one about $\Div_\mathcal{X}(D)$ from Proposition \ref{prop:Div-SP} and the fact that $\inscrit_\mathcal{X}(D)=\emptyset$.

Suppose now that $F\subseteq K$ is any fractional subset, and suppose that $dF\subseteq K$, $d\neq 0$. Then, $\Int(F,D)\simeq\Int(dF,D)$ \cite[Remark I.1.11]{intD} and $\Int(dF,D)$ is an overring of $D$ since $dF\subseteq D$ \cite[Proposition I.1.6]{intD}. Hence, $\Int(F,D)$ is a CSD-domain (Corollary \ref{cor:CSD-overring}), and the other claims follow from the fact that, if $D=V$ is a valuation domain, then the space of unitary prime ideals of $\Int(F,V)$ is homeomorphic to $\widehat{F}$ \cite[Proposition 5.2.2]{intD}.
\end{proof}

\bibliographystyle{plain}
\bibliography{/bib/articoli,/bib/libri,/bib/miei}
\end{document}